\documentclass{article}
\usepackage[utf8]{inputenc}
\usepackage{amsmath,amsthm,amssymb,enumerate,kantlipsum,graphicx,amsfonts}
\usepackage{float}
\usepackage[colorlinks=true,citecolor=blue,urlcolor=blue,linkcolor=black]{hyperref}
\usepackage{tikz-cd}
\usepackage{setspace}
\usepackage{braket}
\usepackage{parskip}
\usepackage[british]{babel}
\usepackage[autostyle]{csquotes}
\usepackage{microtype}
\usepackage{setspace}
\usepackage[a4paper, total={6in, 8in}]{geometry}
\usetikzlibrary{decorations.markings,graphs,graphs.standard}

\newcommand{\soc}[1]{\mathrm{soc}(#1)}
\newcommand{\rank}[1]{\mathrm{rank}(#1)}

\newcommand{\qbinom}[2]{\genfrac{[}{]}{0pt}{}{#1}{#2}}

\newcommand{\F}{\mathbb{F}}

\newcommand{\aut}[1]{\mathrm{Aut}(#1)}

\newcommand{\PGL}{\mathrm{PGL}}
\newcommand{\PSL}{\mathrm{PSL}}

\newcommand{\Pommega}{\mathrm{P}\Omega}
\newcommand{\SL}{\mathrm{SL}}
\newcommand{\PSp}{\mathrm{PSp}}

\newcommand{\Z}{\mathbb{Z}}

\newcommand{\U}{\mathrm{PSU}}
\newcommand{\bsum}{\displaystyle\sum}
\newtheorem{thm}{Theorem}[section]

\newtheorem{lemma}[thm]{Lemma}

\newtheorem{prop}[thm]{Proposition}
\newtheorem{orthprop}{Proposition}[section]
\newtheorem{orthlemma}[orthprop]{Lemma}

\theoremstyle{definition}
\newtheorem{rmks}[thm]{Remarks}
\newenvironment{remark}
  {\par\addvspace{\medskipamount}\noindent\textit{Remark.}\ \ignorespaces}
  {\par\addvspace{\medskipamount}}
\newtheorem{defi}[thm]{Definition}
\newtheorem{orthdefi}[orthprop]{Definition}
\newtheorem{ex}[thm]{Example}

\title{Strongly Regular Graphs of Rank Four}
\author{William H. Allen}
\date{}

\begin{document}

\maketitle 

\begin{abstract}
    \noindent Strongly regular graphs are regular graphs with a constant number of common neighbours between adjacent vertices, and a constant number of common neighbours between non-adjacent vertices. These graphs have been of great interest over the last few decades and often give rise to interesting groups of automorphisms. In this paper we take a reverse approach, and leverage strong classification results on rank four permutation groups to classify the strongly regular graphs admitting a non-affine primitive rank four group of automorphisms.
\end{abstract}

\section{Introduction}
A \emph{regular graph} is a graph on $v$ vertices, such that each vertex has fixed degree $k$. If $\Gamma$ is a regular graph such that every pair of adjacent vertices has $\lambda$ common neighbours, and every pair of non-adjacent vertices has $\mu$ common neighbours, then we say that $\Gamma$ is a \emph{strongly regular graph} with the parameters $(v,k,\lambda, \mu)$. Let $G$ be a group acting transitively on a finite set $\Omega$, so that the action of $G$ on $\Omega$ induces an action of $G$ on $\Omega\times \Omega$ with $r$ orbits (called orbitals); we say that $G$ has \emph{rank} $r$. For each orbital $\Delta$, there is an orbital $\Delta^*$, called the \emph{paired orbital}, where $(\alpha,\beta)\in \Delta^*$ if and only if $(\beta,\alpha)\in \Delta$. If $\Delta=\Delta^*$, one says that $\Delta$ is \emph{self paired}, and the orbital $\{(x,x) : x\in \Omega\}$ is called the \emph{diagonal orbital}. For a non-diagonal orbital $\Delta$, its \emph{orbital graph} is the undirected graph with vertex set $\Omega$ and edge relation $\Delta\cup\Delta^*$. All rank conditions refer to the specified permutation action of $G$, which need not be the full automorphism group. A well known result of D.G. Higman says that the graphs with edge sets $\Delta\cup\Delta^*$, for all non-diagonal orbitals $\Delta$, are connected if and only if the $G$-action on $\Omega$ is primitive. If $G$ acts primitively on $\Omega$ with rank $3$ and has even order, then $G$ has three orbitals $\Delta_0$, $\Delta_1$, $\Delta_2$, the latter two non-diagonal and self-paired. Indeed, an involution interchanges two points, so one non-diagonal orbital is self-paired, and hence so is the other. The orbital graphs $(\Omega, \Delta_1)$ and $(\Omega, \Delta_2)$ are a complementary pair of strongly regular graphs.

A strongly regular graph may admit a primitive group of automorphisms of rank larger than three. We classify the non-affine primitive groups of rank four and determine the strongly regular graphs admitting such groups, up to complementation. Since a rank four group has three non-diagonal orbitals, either the edge relation or its complement is a single self-paired orbital. We choose these single-orbital representatives throughout the classification. Before stating our main result, we give some examples, found in \cite[3.1.6, 3.1.4, 3.2.4, 4.8]{srg}.

\begingroup
\setlength{\topsep}{\medskipamount}
\begin{ex}\label{ex:graph-families}
    \begin{enumerate}[(i)]
        \item \textbf{Nonisotropic unitary graphs:} Let $n\ge 3$, and $V$ be an $n$-dimensional vector space over $\F_{q^2}$, where $q$ is a prime power, and $h:V\times V\to \F_{q^2}$ an associated, non-degenerate hermitian form. Define $NU_n(q)$ to be the graph whose vertices are the non-singular $1$-spaces of $V$, where two vertices are joined by an edge if and only if their joining projective line is tangent to the hermitian variety; that is, the projective line passing through the two projective points meets the hermitian variety $H = \{\braket{v} : h(v,v)=0\}$ in precisely one point. The graph $NU_n(q)$ is strongly regular with parameters $(v,k,\lambda,\mu)$ given by:
        \begin{align*}
            v &= \frac{q^{n-1}(q^n-(-1)^n)}{q+1} \\
            k &= (q^{n-1}+(-1)^n)(q^{n-2}-(-1)^n) \\
            \lambda &= q^{2n-5}(q+1)-(-1)^nq^{n-2}(q-1) - 2 \\
            \mu &= q^{n-3}(q+1)(q^{n-2}-(-1)^n).
\end{align*}

        \item \textbf{Nonisotropic orthogonal graphs:} Let $n=2m+1$ with $m\ge 2$, and $V$ be an $n$-dimensional vector space over $\F_q$, with $q$ odd, and $Q:V\to \F_q$ an associated non-degenerate quadratic form. For $\epsilon=\pm$, let $\Omega^\epsilon$ be the set of nonsingular $1$-spaces with perpendicular space of type $\mathrm{O}_{2m}^\epsilon(q)$. Define $NO_n^\epsilon(q)$ to be the graph on $\Omega^\epsilon$, where two vertices are joined by an edge if and only if their joining projective line is tangent to the quadric of $V$.
        The graph $NO_n^\epsilon(q)$ is strongly regular with parameters
        \begin{align*}
            v & = \frac{1}{2}q^m(q^m+\epsilon) \\
            k &= (q^{m-1}+\epsilon)(q^m-\epsilon) \\
            \lambda &= 2(q^{2m-2}-1)+\epsilon q^{m-1}(q-1) \\
            \mu &= 2q^{m-1}(q^{m-1}+\epsilon).
    \end{align*}
    On the same vertex set, define $NO_{2m+1}^{\epsilon,\perp}(q)$ by joining distinct nonsingular $1$-spaces when they are orthogonal. For $q=5$, this graph is strongly regular with parameters
    \[
        \left(\frac{5^m(5^m+\epsilon)}{2},
        \frac{5^{m-1}(5^m-\epsilon)}{2},
        \frac{5^{m-1}(5^{m-1}+\epsilon)}{2},
        \frac{5^{m-1}(5^{m-1}-\epsilon)}{2}\right).
    \]
    \item \textbf{Eight dimensional orthogonal polar graphs:} Let $K = \Pommega_8^+(q)$, and $P_1$ be a maximal parabolic subgroup stabilising a singular $1$-space. The group $K$ acts on the cosets $K/P_1$ with rank $3$, and one of the orbital graphs is the polar graph (where two vertices are joined whenever perpendicular), whose complement $\Gamma$ is strongly regular with parameters $((q^3+1)(q^2+1)(q+1), q^6, q^2(q-1)(q^3-1), q^5(q-1))$. The group $K$ has a subgroup $G=\Omega_7(q)$ acting on the cosets $G/P_3$ with rank $4$, which has $\Gamma$ as an orbital graph \cite[3.2.4]{srg}. 
    
    \item \textbf{Seven dimensional orthogonal polar graphs:} Let $q$ be any prime power, let $K=\Omega_7(q)$, and let $P_1$ be the maximal parabolic stabilising a singular $1$-space of the underlying seven-dimensional orthogonal space. The action of $K$ on the cosets $K/P_1$ has rank $3$, and one of the orbital graphs is the polar graph; the complement of which, $\Gamma$, is strongly regular with parameters $(\frac{q^6-1}{q-1},q^5,q^4(q-1),q^4(q-1))$. The group $K$ has a subgroup $G=G_2(q)$ acting on $G/P_1$ with rank $4$, which has $\Gamma$ as one of its orbital graphs \cite[3.2.4]{srg}. Here $P_1$ in $G_2(q)$ is obtained by deleting the short simple root in the Dynkin diagram.

    \item \textbf{Distance three graphs of symplectic dual polar graphs:} Let $G = \PSp_6(q)$, $q$ be an odd prime power, and $P_3$ be a maximal parabolic stabilising a totally isotropic $3$-space. The group $G$ acts on the cosets $G/P_3$ with rank four, and one of the orbital graphs is a dual polar graph. The distance three graph of this dual polar graph (where adjacency is given by joining two vertices at distance three) is strongly regular with parameters $((q^3+1)(q^2+1)(q+1),q^6,q^2(q-1)(q^3-1),q^5(q-1))$ \cite[3.2.4]{srg}.

    \item \textbf{Distance three graphs of} $G_2(q)$\textbf{ actions:} Let $G = G_2(q)$ with $q\ne 3^a$, and $P_2$ be the maximal parabolic obtained by deleting the long simple root in the Dynkin diagram. Let $G_{2,2}$ denote the distance transitive orbital graph corresponding to the action of $G$ on $G/P_2$ as defined in \cite[Table 10.8]{drg}.  The distance three graph of $G_{2,2}$ is strongly regular with parameters $(\frac{q^6-1}{q-1},q^5,q^4(q-1),q^4(q-1))$. 
    \end{enumerate}
\end{ex}
\endgroup
\vspace{2mm}
\begin{thm}\label{thm:rank-four-srg}
    Let $\Gamma$ be a strongly regular graph with a group of automorphisms $G\le\aut{\Gamma}$ and point stabiliser $H$, such that $G$ is a non-affine primitive permutation group of rank four, and suppose that $\Gamma$ is an orbital graph for $G$. Then, after replacing $\Gamma$ by its complement if necessary, its edge relation is a single self-paired orbital of $G$, and one of the following holds:
    \begin{enumerate}[(i)]
        \item The graph $\Gamma$ is $NU_n(3)$ or $NU_n(4)$, with $n\ge3$, and $G\triangleright \U_n(3)$ or $\U_n(4)$, respectively, in its action on nonsingular $1$-spaces.
        \item The group $G\triangleright\Pommega_{2m+1}(5)$ is in its action on nonsingular $1$-spaces of a fixed type, and $\Gamma$ is either $NO_{2m+1}^{\pm}(5)$ or the orthogonality graph $NO_{2m+1}^{\epsilon,\perp}(5)$ on the same vertex set.
        \item The graph $\Gamma$ is the complement of the polar graph of $O_8^+(q)$, and $G\triangleright \Pommega_7(q)$ is in its action on $\Pommega_7(q)/P_3$.
        \item The graph $\Gamma$ is the complement of the polar graph of $O_7(q)$, and $G\triangleright G_2(q)$ is in its action on $G_2(q)/P_1$.
        \item The graph $\Gamma$ is the distance-$3$ graph of the dual polar graph of $\PSp_6(q)$ with $q$ odd, and $G\triangleright \PSp_6(q)$ is in its action on $\PSp_6(q)/P_3$.
        \item The graph $\Gamma$ is the distance-$3$ graph of $G_{2,2}(q)$, and $G\triangleright G_2(q)$ is in its action on $G_2(q)/P_2$, where $q\ne 3^a$ for a positive integer $a$.
        \item The socle $S=\soc{G}$, the subgroup $S\cap H$, and the parameters of $\Gamma$ belong to Table~\ref{tab:rank-four-srg-exceptions}, with the graph multiplicity indicated there.
    \end{enumerate}
\end{thm}
\begin{table}[p]
\begin{center}
    \setlength{\tabcolsep}{4pt}
    \begin{tabular}{|c|c|c|l|}
    \hline
    $S=\soc{G}$ & $S\cap H$ & Parameters & Remarks \\
    \hline
    $A_6$ & $D_{10}$ & $(36,10,4,2)$ & \\
    $A_7$ & $A_7\cap (S_3\times S_4)$ & $(35,18,9,9)$ & \\
    $A_{10}$ & $A_{10}\cap (S_3\times S_7)$ & $(120,63,30,36)$ & \\
    \hline
    $\PSL_2(8)^2$ & $D_{18}\times D_{18}$ & $(784,54,26,2)$ & \\
    $\PSL_2(8)^2$ & $D_{18}\times D_{18}$ & $(784,243,82,72)$ & \\
    $\PSL_2(8)^2$ & $D_{18}\times D_{18}$ & $(784,486,298,306)$ & \\
    $\PSL_2(16)$ & $D_{34}$ & $(120,68,40,36)$ & \\
    $\PSL_3(4)$ & $P_{1,2}$ & $(105,32,4,12)$ & \\
    $\PSL_3(4)$ & $\PSL_3(2)$ & $(120,42,8,18)$ & \\
    \hline
    $\U_3(3)$ & $\PSL_2(7)$ & $(36,21,12,12)$ & \\
    $\U_3(5)$ & $A_6.2$ & $(175,72,20,36)$ & \\
    $\Pommega_{8}^+(3)$ & $\Pommega_{8}^+(2)$ & $(28431,3150,621,315)$ & \\
    $\Pommega_{8}^+(3)$ & $\Pommega_{8}^+(2)$ & $(28431,2880,324,288)$ & \\
    $G_2(3)$ & $\SL_3(3).2$ & $(378,117,36,36)$ & \\
    $G_2(4)$ & $\SL_3(4).2$ & $(2080,1008,480,496)$ & \\
    $G_2(5)$ & $\mathrm{SU}_3(5).2$ & $(7750,1575,300,325)$ & \\
    $G_2(5)$ & $\mathrm{SU}_3(5).2$ & $(7750,3024,1148,1200)$ & \\
    $G_2(8)$ & $\mathrm{SL}_3(8).2$ & $(131328,98112,73216,73536)$ & \\
    $^2F_4(2)'$ & $\PSL_3(3).2$ & $(1600,351,94,72)$ & \\
    \hline
    $M_{11}$ & $S_{5}$ & $(66,20,10,4)$ & \\
    $M_{12}$ & $\PSL_{2}(11)$ & $(144,22,10,2)$ & \\
    $M_{12}$ & $\PSL_{2}(11)$ & $(144,55,22,20)$ & \\
    $M_{12}$ & $\PSL_{2}(11)$ & $(144,66,30,30)$ & $2$ graphs \\
    $M_{22}$ & $2^{4}{:}S_{5}$ & $(231,30,9,3)$ & \\
    $M_{22}$ & $2^{4}{:}S_{5}$ & $(231,40,20,4)$ & \\
    $M_{23}$ & $M_{11}$ & $(1288,792,476,504)$ & \\
    $J_{2}$ & $3.\PGL_{2}(9)$ & $(280,36,8,4)$ & \\
    $J_{2}$ & $3.\PGL_{2}(9)$ & $(280,135,70,60)$ & \\
    $\operatorname{Fi}_{22}$ & $\Pommega_{8}^+(2){:}S_{3}$ & $(61776,1575,198,36)$ & \\
    $\operatorname{Fi}_{22}$ & $\Pommega_{8}^+(2){:}S_{3}$ & $(61776,22400,8224,8064)$ & \\
    \hline
\end{tabular}
\caption{The exceptional rank four strongly regular graphs, up to complementation.}
\label{tab:rank-four-srg-exceptions}
\end{center}
\end{table}
\vspace{2mm}
\begin{rmks}
    \begin{enumerate}[(i)]
        \item All of the groups $G$ in Table~\ref{tab:rank-four-srg-exceptions} are almost simple, except for the three entries with socle $\PSL_2(8)^2$. In these cases $G\cong\PSL_2(8)^2.6\le(\PSL_2(8).3)\wr S_2$ acts in product action of degree $28^2$, and the adjacency relations of the three graphs are the non-diagonal relations of the Mathon scheme \cite[Proposition 8.11.2]{srg}. 
        \item It can be the case that $\aut{\Gamma}$ has rank $3$; this happens in cases (iii) and (iv) of Theorem~\ref{thm:rank-four-srg} and for the groups with $\soc{G}\cong G_2(r)$ with $r\in \{3,4\}$ in Table~\ref{tab:rank-four-srg-exceptions}. In the latter cases, $\aut{\Gamma}$ has rank $3$ with socle $\Omega_7(r)$.
        \item In the cases in Table~\ref{tab:rank-four-srg-exceptions} where $S=A_n$, with $n\in \{7,10\}$, the group $G$ acts on $\Omega_3$, where $\Omega_3$ is the set of $3$-element subsets of $\{1,\ldots,n\}$. For $0\le i\le3$, the orbitals are given by
\[
    \Delta_i := \{(A,B)\in \Omega_3\times\Omega_3 : |A\cap B| = 3-i\}.
\]
    The graph $J(n,3,j)$, in which adjacency means $|A\cap B|=j$, is called a \emph{generalised Johnson graph}. The graphs in Table~\ref{tab:rank-four-srg-exceptions} are $J(n,3,1)$.
        \item The socle $S=G_2(2)'\cong\U_3(3)$ has rank four in the degree-$63$ $P_2$ action, with stabiliser $N_1\cong4.S_4$, giving $NU_3(3)$ in Theorem~\ref{thm:rank-four-srg} (i); the full-group action is also in (vi). In the $P_1$ action its stabiliser is $4^2:S_3$ and its rank is five, so $G_2(2)\cong S:2$ is minimal of rank four, giving (iv).
        \item Both degree-$144$ actions of $M_{12}:2$ give nonisomorphic graphs with parameters $(144,66,30,30)$: their adjacency matrices have ranks $46$ and $22$ over $\F_2$.
    \end{enumerate}
\end{rmks}
\vspace{2mm}
\begin{samepage}
In Section~\ref{sec:graph-notation}, we define all necessary notation involved in the proof, in Section~\ref{sec:group-classification}, we state the classification of non-affine rank four groups, and in Section~\ref{sec:computational-tools}, we cover all computational tools used. Section~\ref{subsec:orthogonal-intersection numbers} provides the orthogonal intersection numbers used in Section~\ref{sec:main-proof}, which is dedicated to the proof of Theorem~\ref{thm:rank-four-srg}. 
\par
\end{samepage}

\textbf{Acknowledgements:} The author would like to thank the Heilbronn Institute for the funding of the author's PhD Thesis, which the research conducted in this paper is a part of. A special appreciation is given to the author's supervisor, Professor Martin Liebeck, who has been of great help with the production of this paper. Finally, the author would like to acknowledge Professor L.H Soicher, for his help in Section~\ref{sec:computational-tools}.

\section{Preliminaries}

\subsection{Distance regular graphs and notation}\label{sec:graph-notation}
A finite, simple, connected graph $\Gamma$, is said to be \emph{distance regular} with \emph{parameters} $a_i$, $b_i$, and $c_i$ if for any two vertices $x$, $y$ with $d(x,y)=i$, the number of vertices $z$, which are adjacent to $y$, and with $d(x,z)=i-1$, $i$, and $i+1$, are respectively $c_i$, $a_i$, and $b_i$, where $d : \Gamma^2\to \Z$ is the distance function. Let $l$ be the diameter of $\Gamma$. The \emph{intersection array} $\{b_0,\ldots,b_{l-1} ; c_1,\ldots,c_l\}$, is the symbol which suffices to obtain all parameters. A distance regular graph with valency $k=b_0$ satisfies $a_i+b_i+c_i = k$ for all $i$. Given a vertex $x$ in a distance regular graph, one writes $k_i$ to denote the number of vertices at distance $i$ from $x$, where $k_0=1$, and $k_{i+1}=k_ib_i/c_{i+1}$. The graph $\Gamma$ is said to be \emph{distance transitive} if, for any pairs of vertices $(x_0,y_0)$, and $(x_1,y_1)$ with $d(x_0,y_0)=d(x_1,y_1)$, there is a graph automorphism taking one pair to the other.

Given a primitive permutation group acting on a set $\Omega$ with rank $r+1$, let $\Delta_0$, $\ldots$, $\Delta_r$ be its orbitals and take vertices $x$ and $y$, with $(x,y)\in \Delta_h$, where $h\in\{0,\ldots,r\}$. We define the \emph{intersection number} $p_{ij}^h$ to be
\[
    p_{ij}^h :=\#\{z\in \Omega : (x,z) \in \Delta_i, \, (z,y)\in\Delta_j\}.
\]
\begingroup
\setlength{\topsep}{\medskipamount}
\begin{lemma}\label{lem:example}
    \textup{(\cite[Proposition 1.3.1]{srg})}. Suppose that all the orbitals are self-paired, and let $\Delta_0$ denote the diagonal orbital. Write $k_j$ for the corresponding suborbit lengths and $v=|\Omega|$. The intersection numbers $p_{ij}^h$ satisfy the following relations:
    \begin{enumerate}[(i)]
        \item $p_{0j}^h=\delta_{jh}$, $p_{ij}^0=\delta_{ij}k_j$, $p_{ij}^h=p_{ji}^h$,
        \item $\bsum_i\, p_{ij}^h = k_j$, $\bsum_j\, k_j = v$,
        \item $p_{ij}^hk_h = p_{ih}^jk_j$,
        \item $\bsum_l\, p_{ij}^lp_{hl}^m = \bsum_{l}\, p_{hj}^lp_{il}^m$.
    \end{enumerate}
\end{lemma}
\endgroup
In particular for $0\le i,j\le l$, when $\Gamma$ is distance transitive and $\Delta_i$ is the distance-$i$ relation, we have
\[
    p_{1j}^i = \begin{cases}
        c_i & \text{if } j = i-1 \\
        a_i & \text{if } j = i \\
        b_i & \text{if } j = i+1 \\
        0 & \text{otherwise}.
    \end{cases}
\]
Thus the intersection array may be rewritten in terms of the intersection numbers. 

For groups of Lie type, we use Bourbaki numbering of the nodes of Dynkin diagrams \cite[Chapter~VI, Plates~I--IX]{bourbaki}, with the folded relative diagrams for twisted groups as described in \cite[pp.~23--26]{allen-thesis}, and write $P_I$ for the parabolic subgroup obtained by deleting the nodes in $I$. If $G$ is a classical group with associated vector space $V$, and form $\kappa$, then we write $N_1$ to denote the stabiliser of a nonsingular $1$-space with respect to the form $\kappa$. If $G=\mathrm{O}_{2m+1}(q)$ with $q$ odd and $v$ a nonsingular vector in $V$, the stabiliser of $\braket{v}$ is denoted by $N_1^\epsilon$, with $\epsilon=\pm$, where $v^\perp$ is of type $\mathrm{O}_{2m}^\epsilon (q)$.

\subsection{The classification of non-affine rank four groups}\label{sec:group-classification}
Let $\mathcal{G}$ be the set of non-affine primitive permutation groups of rank four. The classification below builds on Cuypers' classification of primitive almost simple groups of Lie type with actions of rank at most five \cite{cuypers}, Vauhkonen's rank four classification for linear socles \cite{vauk}, and the work of Muzychuk and Spiga on alternating and sporadic socles \cite{muzy}. Determining the individual ranks, the required extensions and their orbital structures requires further analysis. This analysis, together with the non-almost-simple cases, is established in the author's PhD thesis \cite[Chapter~1]{allen-thesis}; a fuller account of primitive non-affine groups of ranks three, four and five will appear in a forthcoming paper.

\begingroup
\setlength{\topsep}{\medskipamount}
\begin{defi}
    Let $G$ be a finite primitive permutation group on $\Omega$, with socle $N$. We say that $G$ is a \emph{minimal extension of $N$ of rank $4$} if $\rank{G,\Omega}=4$, and there is no proper subgroup $K<G$ such that
    \[
        N\le K<G
        \quad\text{and}\quad
        \rank{K,\Omega}=4.
    \]
\end{defi}

We can now state the classification of the non-affine primitive permutation groups of rank four.

\begin{thm}\label{thm:rank-four-groups}
    Suppose that $G\in \mathcal{G}$, the set of non-affine rank $4$ primitive groups. Then one of the following holds:
    \begin{enumerate}[(i)]
        \item The group $G$ is of simple diagonal type, and has socle isomorphic to $A_5\times A_5$.
        \item The group $G$ has socle $T^2:=\PSL_2(8)^2$, and $T^2\le G\le K\wr S_2$ is of product type, where $K\cong \PSL_2(8).3$ acts $2$-transitively on $28$ points and $G\cong \PSL_2(8)^2.6$.
        \item There is an almost simple group $K$, acting $2$-transitively on a set $\Delta$, with socle $T$, and the group $G$ acts in its natural product action on $\Omega =  \Delta^3$ with $T^3\le G\le K\wr S_3$.
        \item The group $G$ is almost simple with point stabiliser $H$ and socle $S$, such that $S$ and $H\cap S$ belong to Table~\ref{tab:rank-four-families} or Table~\ref{tab:rank-four-exceptions}, subject to the restrictions there; the $G$ column records a representative minimal extension of rank four.
    \end{enumerate}
    All orbitals of $G$ are self-paired, except when $G=\U_3(3)$ is in its degree-$36$ action on $\PSL_2(7)$-cosets.
\end{thm}
\endgroup

In Tables~\ref{tab:rank-four-families} and~\ref{tab:rank-four-exceptions}, the column $G$ gives a representative minimal extension of $S$ attaining rank four, with $G=S$ whenever the socle itself has rank four. Such an extension need not be unique, and an arbitrary rank-four group need not contain the representative displayed.

In the tables, $\tau$ denotes a graph involution, $\rho$ a triality automorphism of order three, and $\sigma$ a Frobenius automorphism. For $\Pommega_{2m}^{-}(4)$, $\eta$ denotes a semilinear automorphism of order four whose square induces the graph involution. The dot notation does not imply that the extension splits.

\begin{table}[ht]
    \centering
    \begin{tabular}{|c|c|c|c|}
    \hline
       $S$ & $H\cap S$ & $G$ & Restrictions \\
    \hline
    $A_n$ & $A_n\cap (S_3\times S_{n-3})$ & $S$ & $n\ge 7$ \\
    \hline
    $\PSL_n(q)$ & $P_3$ & $S$ & $n \ge 6$ \\
    $\PSL_3(q)$ & $P_{1,2}$ & $S.\braket{\tau}$ & $q\ge2$ \\
    \hline
    $\PSp_6(q)$ & $P_3$ & $S$ & $q$ odd \\
       \hline
    $\U_n(3)$ & $N_1$ & $S$ & $n\ge3$ \\
    $\U_n(4)$ & $N_1$ & $S.\braket{\sigma}$ & $n\ge3$ \\
    $\U_n(q)$ & $P_3$ & $S$ & $n\in\{6,7\}$ \\
       \hline
    $\Pommega_{7}(q)$ & $P_3$ & $S$ &  \\
    $\Pommega_8^-(q)$ & $P_3$ & $S$ &  \\
    $\Pommega_{2m}^+(q)$ & $P_{m-1}$, $P_m$ & $S$ & $m\in\{6,7\}$ \\
    $\Pommega_{2m+1}(5)$ & $N_1^{\pm}$ & $S$ & $m\ge 2$ \\
    $\Pommega_{2m}^{\pm}(5)$ & $N_1$ & $S$ & $m\ge3$ \\
    $\Pommega_{2m}^{+}(4)$ & $N_1$ & $S.\braket{\sigma}$ & $m\ge3$ \\
    $\Pommega_{2m}^{-}(4)$ & $N_1$ & $S.\braket{\eta}$ & $m\ge3$ \\
       \hline
    $G_2(q)$ & $P_1$, $P_2$ & $S$ & $q\ge3$ \\
    $E_7(q)$ & $P_7$ & $S$ &  \\
    $^3D_4(q)$ & $P_1$, $P_2$ & $S$ &  \\
       \hline
    \end{tabular}
    \caption{The families of $G\in \mathcal{G}$ with $G$ almost simple}
    \label{tab:rank-four-families}
\end{table}

\begin{table}[!h]
    \centering
    \begin{tabular}{|c|c|c|c|}
    \hline
       $S$ & $H\cap S$ & $G$ & Restrictions \\
    \hline
    $\PSL_2(7)$ & $D_8$ & $\aut{S}$ &  \\
    $\PSL_2(8)$ & $D_{18}$ & $S$ &  \\
    $\PSL_2(q)$ & $D_{2(q+1)}$ & $\aut{S}$ & $q\in\{16,32\}$ \\
    $\PSL_2(16)$ & $A_5$ & $\aut{S}$ &  \\
    $\PSL_2(25)$ & $S_5$ & $S$ &  \\
    $\PSL_3(4)$ & $\PSL_3(2)$ & $S$ &  \\
    $\PSL_2(19)$ & $A_5$ & $S$ &  \\
    $\PSp_6(4)$ & $G_2(4)$ & $S.\braket{\sigma}$ &  \\
    $\U_3(3)$ & $\PSL_2(7)$ & $S$ &  \\
    $\U_3(3)$ & $4^2:S_3$ & $\aut{S}$ &  \\
    $\U_6(2)$ & $\PSp_6(2)$ & $S$ &  \\
    $\U_3(5)$ & $A_6.2$ & $S$ &  \\
    $\Pommega_{7}(5)$ & $G_2(5)$ & $S$ &  \\
    $\Pommega_{7}(3)$ & $\PSp_6(2)$ & $S$ &  \\
    $\Pommega_{8}^+(2)$ & $A_9$ & $S$ &  \\
    $\Pommega_8^+(3)$ & $\Pommega_8^+(2)$ & $S.\braket{\rho}$ &  \\
    $G_2(q)$ & $\SL_3(q).2$ & $S$ & $q\in\{3,4\}$ \\
    $G_2(8)$ & $\SL_3(8).2$ & $S.\braket{\sigma}$ &  \\
    $G_2(5)$ & $\mathrm{SU}_3(5).2$ & $S$ &  \\
    $^2F_4(2)'$ & $\PSL_3(3).2$ & $S$ &  \\
    \hline
    $M_{11}$ & $S_5$ & $S$ &  \\
    $M_{12}$ & $\PSL_2(11)$ & $S.2$ & $2$ actions \\
    $M_{22}$ & $2^4:S_5$ & $S$ &  \\
    $M_{23}$ & $A_8$ & $S$ &  \\
    $M_{23}$ & $M_{11}$ & $S$ &  \\
    $M_{24}$ & $2^4:A_8$ & $S$ &  \\
    $M_{24}$ & $2^6:3.S_6$ & $S$ &  \\
    $Co_1$ & $Co_2$ & $S$ &  \\
    $J_2$ & $3.\PGL_2(9)$ & $S$ &  \\
    $McL$ & $M_{22}$ & $S$ & $2$ actions \\
    $He$ & $\PSp_4(4):2$ & $S.2$ &  \\
    $Fi_{22}$ & $\Pommega_8^+(2):S_3$ & $S$ &  \\
    \hline
    $A_6$ & $D_{10}$ & $M_{10}$ &  \\
    $A_{12}$ & $M_{12}$ & $S$ & $2$ actions \\
    $A_{2r}$ & $A_{2r}\cap (S_r\wr S_2)$ & $S$ & $r\in \{6,7\}$ \\
    \hline
    $^2E_6(2)$ & $F_4(2)$ & $S$ &  \\
    \hline
    \end{tabular}
    \caption{The exceptional socle/stabiliser pairs for almost simple $G\in\mathcal{G}$}
    \label{tab:rank-four-exceptions}
\end{table}

\subsection{Computational tools}\label{sec:computational-tools}
In the case where $G$ belongs to Table~\ref{tab:rank-four-exceptions}, it is possible to evaluate the parameters of the associated orbital graphs computationally; in each such case, the computer algebra system GAP \cite{gap} is used. Using the GRAPE package functions \cite[2.8]{grape}, \cite[4.3]{grape} in GAP, one may check whether the orbital graphs of $G$ are distance regular; if such an orbital graph is distance regular of diameter two, then it is strongly regular. In the last entry of Table~\ref{tab:rank-four-exceptions} where $S = \,^2E_6(2)$ and $S\cap H = F_4(2)$, the desired coset action was constructed by L.H. Soicher. Using these methods, we identify from Table~\ref{tab:rank-four-exceptions} the orbital graphs which are strongly regular; up to complementation, these are recorded in Table~\ref{tab:rank-four-srg-exceptions}, unless already covered by the infinite families in Theorem~\ref{thm:rank-four-srg}. For example, the degree $300$ action of $\PSp_4(5)\cong\Pommega_5(5)$ with stabiliser $\PSL_2(25).2\cong N_1^-$ is covered by Theorem~\ref{thm:rank-four-srg} (ii). Selected symbolic intersection number calculations and their factored expressions, including those obtained from the parabolic recurrence~\eqref{eq:parabolic-recurrence}, were calculated using \emph{Siderea}~\cite{siderea}, a forthcoming SageMath~\cite{sage} package containing classification and permutation data for primitive non-affine groups of ranks three, four and five, together with tools for computing their associated Schurian association schemes.

\clearpage
\section{Intersection numbers for orthogonal \texorpdfstring{$N_1$}{N1} actions}
\label{subsec:orthogonal-intersection numbers}

The nonsingular one-space actions of orthogonal groups, and their association schemes, were studied by Bannai, Shen and Song \cite{orthogonal}. Their results express the character tables and intersection parameters in terms of controlling lower-dimensional schemes: in even dimension these are the $\mathrm{SL}_2(q)$ schemes for even $q$ and the $\mathrm{PSL}_2(q)$ schemes for odd $q$. In odd dimension and odd characteristic, their Theorems~6.3 and~7.3 reduce the character tables for the two square classes of nonsingular points to those of the corresponding three-dimensional orthogonal schemes. As noted in \cite[Remark~3]{orthogonal}, the controlling entries are not evaluated explicitly there; Kwok \cite{kwok} subsequently computed the controlling character tables in terms of character sums.

Although the intersection numbers can in principle be recovered from these character tables, we give direct geometric counts in the relation labelling used here. The resulting explicit formulae cover both characteristics and, in odd characteristic, both even and odd dimensions. They are used in Section~\ref{sec:main-proof} to compare the common-neighbour counts required for the rank four graph classification.

\begingroup
\setlength{\parskip}{0.5\baselineskip plus 2pt minus 1pt}
\setlength{\abovedisplayskip}{8pt plus 2pt minus 2pt}
\setlength{\belowdisplayskip}{8pt plus 2pt minus 2pt}
\setlength{\abovedisplayshortskip}{4pt plus 2pt minus 1pt}
\setlength{\belowdisplayshortskip}{6pt plus 2pt minus 2pt}
\makeatletter
\def\thm@space@setup{%
  \thm@preskip=\parskip
  \thm@postskip=\parskip
}
\makeatother

\subsection{Common setup and notation}
\label{orth:subsec:common-setup}
Let $(V,Q)$ be a nondegenerate quadratic space over $\F_q$, with polar form
\[
    b(u,v)=Q(u+v)-Q(u)-Q(v).
\]
We take $\dim V=2m\ge6$ in even dimension, and $\dim V=2m+1\ge5$ with $q$ odd in odd dimension. Decompose $V$ as
\[
    V=H_1\perp\cdots\perp H_r\perp V_0,
\]
where $H_i=\langle e_i,f_i\rangle$ is hyperbolic, with $Q(e_i)=Q(f_i)=0$ and $b(e_i,f_i)=1$, and $V_0$ is anisotropic. We use a basis adapted to this orthogonal decomposition. In plus type it is $e_1,\ldots,e_m,f_1,\ldots,f_m$, with $r=m$ and $V_0=0$; in minus type it is $e_1,\ldots,e_{m-1},f_1,\ldots,f_{m-1},u,v$, with $r=m-1$ and $V_0=\langle u,v\rangle$ an anisotropic plane; in odd dimension (neutral type) it is $e_1,\ldots,e_m,f_1,\ldots,f_m,t$, with $r=m$ and $V_0=\langle t\rangle$ an anisotropic line.  The orthogonal sum $H_1\perp H_2$ is the \emph{head space}; the orthogonal sum $H_3\perp\cdots\perp H_r$ is the \emph{tail space}, and the anisotropic component is kept separate. Norm counts on the complement of the head space include both the tail and the anisotropic component.

For odd $q$, let $\Omega$ be the orbit of nonsingular one-spaces for which $Q(v)$ has a fixed class in $\F_q^*/(\F_q^*)^2$, where $v$ is any nonzero representative. Rescale $Q$ so that every point of $\Omega$ has a norm-one representative, giving
\[
    \Omega=\{\langle v\rangle:Q(v)=1\}.
\]
Each point in $\Omega$ has exactly two norm-one representatives, $v$ and $-v$. Let $\Lambda$ consist of $0$ and one representative from each pair $\{a,-a\}$ in $\F_q^*$. For each $\lambda\in\Lambda$, define the relation
\[
    \Delta_\lambda
    =\{(\langle x\rangle,\langle y\rangle):Q(x)=Q(y)=1,
       \ \langle x\rangle\ne\langle y\rangle,\ b(x,y)^2=\lambda^2\}.
\]
For even $q$, every nonsingular one-space has a unique norm-one representative. We take $\Omega$ to be all nonsingular one-spaces. For each $\lambda\in\F_q$, define
\[
    \Delta_\lambda
    =\{(\langle x\rangle,\langle y\rangle):Q(x)=Q(y)=1,
       \ \langle x\rangle\ne\langle y\rangle,\ b(x,y)=\lambda\}.
\]
In either case, write $\Delta_*$ for the diagonal relation; $\Delta_0$ is the relation of distinct orthogonal points.

Under the dimension and characteristic hypotheses above, these are the non-diagonal orbitals of $\mathrm{PSO}(V)$ on $\Omega$ and coincide with those of $\Pommega(V)$; in characteristic two we use the convention $\mathrm{SO}(V)=\mathrm O(V)$. We use the intersection number notation of Section~\ref{sec:graph-notation}, with non-diagonal indices in $\Lambda$ for odd $q$ and in $\F_q$ for even $q$. The diagonal identities in Lemma~\ref{lem:example} apply with diagonal index $*$.

Let $\mathbf1$ denote the indicator function; $\mathbf1_E=\mathbf1(E)$ is $1$ if $E$ holds and $0$ otherwise.

\begin{orthlemma}[{\cite[Theorem~26.6]{comb}}]\label{lem:orth-norm-counts}
    Let $K$ have dimension $2s\ge2$ and orthogonal type $\tau\in\{\pm1\}$ over $\F_q$. The number $N_K(a)$ of vectors $w\in K$ with $Q(w)=a$ is
    \[
        N_K(0)=q^{2s-1}+\tau(q-1)q^{s-1},\qquad
        N_K(a)=q^{2s-1}-\tau q^{s-1}\quad(a\ne0).
    \]
    In the zero-dimensional space, $N_K(0)=1$ and $N_K(a)=0$ for $a\ne0$.
\end{orthlemma}

\subsection{Odd characteristic and even dimension}
Let $\dim V=2m\ge6$, let $q$ be odd, and let $\tau\in\{\pm1\}$ be the type of $V$.

\begin{orthdefi}
    Let $\tau\in\{\pm1\}$.
    For a $(2m-4)$-dimensional orthogonal space of type $\tau$, define
    \[
        A_\tau
        =
        q^{2m-5}+\tau(q-1)q^{m-3},
        \qquad
        B_\tau
        =
        q^{2m-5}-\tau q^{m-3}.
    \]
    Thus, $A_\tau$ is the number of vectors of norm $0$, including the zero vector, while $B_\tau$ is the number of vectors of any fixed nonzero norm, by Lemma~\ref{lem:orth-norm-counts}.
\end{orthdefi}

For odd $q$, let $\chi$ be the quadratic character of $\F_q$; it is $1$ on nonzero squares, $-1$ on nonsquares, and $\chi(0)=0$.

\begin{orthdefi}\label{orth:def:coeffs_and_tails}
    Let $\gamma_1$, $\gamma_2$, and $\lambda$ be fixed representatives in $\Lambda\subseteq \F_q$, where $q$ is odd, and let $\epsilon\in\{\pm1\}$. Further, let
    \[
        \Delta_\epsilon(X)=AX^2+B_\epsilon X+C\in\F_q[X],
    \]
    where
    \[
        A=\lambda^2-4,\qquad
        B_\epsilon=2(2\gamma_1-\epsilon\gamma_2\lambda),\qquad
        C=\gamma_2^2-4.
    \]

    We define the \emph{vanishing tail coefficient}
    \[
        c_{\gamma_1,\gamma_2}^{\lambda}:\{\pm1\}\to\mathbb Z_{\ge 0}
    \]
    by
    \[
        c_{\gamma_1,\gamma_2}^{\lambda}(\epsilon)=
        \begin{cases}
            q(1+\chi(C)), & A=B_\epsilon=0,\\
            q, & A=0,\ B_\epsilon\neq0,\\
            q-\chi(A), & A\neq0,\ D(\Delta_\epsilon)\neq0,\\
            q+\chi(A)(q-1), & A\neq0,\ D(\Delta_\epsilon)=0,
        \end{cases}
    \]
    where $D(\Delta_\epsilon)=B_\epsilon^2-4AC$ denotes the discriminant of $\Delta_\epsilon$.

    We also define the \emph{repeated-root coefficient}
    \[
        \rho_{\gamma_1,\gamma_2}^{\lambda}:\{\pm1\}\to\mathbb Z_{\ge 0}
    \]
    by
    \[
        \rho_{\gamma_1,\gamma_2}^{\lambda}(\epsilon)
        =
        \#\left\{
        a\in\F_q:
        \Delta_\epsilon(a)=0
        \right\}.
    \]
    Equivalently,
    \[
        \rho_{\gamma_1,\gamma_2}^{\lambda}(\epsilon)=
        \begin{cases}
            1+\chi(D(\Delta_\epsilon)), & A\ne0,\\
            1, & A=0,\ B_\epsilon\ne0,\\
            q\mathbf{1}(C=0), & A=B_\epsilon=0.
        \end{cases}
    \]

\end{orthdefi}

\begin{orthlemma}\label{orth:lem:quadratic-character-sum}
    Let $q$ be odd, let $\chi$ be the quadratic character of $\F_q$, extended by $\chi(0)=0$, and let
    \[
        f(X)=aX^2+bX+c\in\F_q[X].
    \]
    Writing $D=b^2-4ac$, one has
    \[
        \bsum_{x\in\F_q}\chi(f(x))
        =
        \begin{cases}
            q\chi(c),&a=b=0,\\
            0,&a=0,\ b\ne0,\\
            -\chi(a),&a\ne0,\ D\ne0,\\
            (q-1)\chi(a),&a\ne0,\ D=0.
        \end{cases}
    \]
    In particular, if $f$ is monic, then
    \[
        \bsum_{x\in\F_q}\chi(f(x))
        =
        \begin{cases}
            q-1,&D=0,\\
            -1,&D\ne0.
        \end{cases}
    \]
\end{orthlemma}
\begin{proof}
    For $a=b=0$, each of the $q$ summands is $\chi(c)$. If $a=0$ and $b\ne0$, the map $x\mapsto bx+c$ permutes $\F_q$, and the sum is zero because the nonzero squares and nonsquares occur equally often.  Suppose that $a\ne0$.  Completing the square and translating the summation variable gives
    \[
        \bsum_{x\in\F_q}\chi(f(x))
        =
        \chi(a)\bsum_{x\in\F_q}
        \chi\left(x^2-\frac{D}{4a^2}\right).
    \]
    If $D=0$, the latter sum is $q-1$.  If $D\ne0$, count the pairs $(x,y)\in\F_q^2$ satisfying
    \[
        y^2=x^2-\frac{D}{4a^2}.
    \]
    Their number is $q+\sum_x\chi(x^2-D/(4a^2))$.  On the other hand, the equation factors as
    \[
        (x-y)(x+y)=\frac{D}{4a^2},
    \]
    and hence has $q-1$ solutions, one for each choice of $x-y\in\F_q^*$.  The required sum is therefore $-1$.
\end{proof}

\begin{orthprop}\label{orth:prop:even_dim_orthog_pijk}
    Let $(V,Q)$ have dimension $n=2m\ge6$ and type $\tau\in\{\pm1\}$, with $q$ odd, and let $\gamma_1,\gamma_2,\lambda\in\Lambda$. Define
    \begin{align*}
        \mathrm{I}_{\gamma_{1}\gamma_{2}}^\lambda&=\mathbf{1}(\gamma_{1}^{2}=4, \gamma_{2}^{2}=\lambda^{2}) + \mathbf{1}(\gamma_{2}^{2}=4,\gamma_{1}^{2}=\lambda^{2}) \\
        Z&=\mathbf{1}(\gamma_{1}=0)+\mathbf{1}(\gamma_{2}=0).
    \end{align*}
    Then the intersection number $p_{\gamma_1\gamma_2}^\lambda$ is given by
    \[
        p_{\gamma_1\gamma_2}^\lambda
        =\frac{(c_{\gamma_1\gamma_2}^\lambda(1)+c_{\gamma_1\gamma_2}^\lambda(-1))A_\tau
        +(2q^2-c_{\gamma_1\gamma_2}^\lambda(1)-c_{\gamma_1\gamma_2}^\lambda(-1))B_\tau
        -2^Z\mathrm I_{\gamma_1\gamma_2}^{\lambda}}{2^Z}.
    \]
\end{orthprop}
\begin{proof}
    Let $(V,Q)$ be an $n$-dimensional vector space over $\F_q$ and $Q$ the associated non-degenerate quadratic form on $V$. We begin with the case of plus type, so that $\tau=1$. Choose a basis $\mathcal{B}^+ = \{e_1,\ldots,e_m,f_1,\ldots,f_m\}$ such that for all $1\le i,j\le m$, $b(e_i,f_j)=\delta_{ij}$, $b(e_i,e_j)=b(f_i,f_j)=Q(e_i)=Q(f_i)=0$. The orbitals $\Delta_\lambda$ for the action are
    \[
        \Delta_\lambda = \{(\braket{x},\braket{y}) : Q(x) = Q(y) = 1,\, \braket{x}\ne\braket{y},\, b(x,y)^2 =\lambda^2 \},
    \]
    for $\lambda\in\Lambda$, where $\Lambda$ consists of $0$ together with one representative from each coset of $\{\pm1\}$ in $\F_q^*$.
    
    Let $x=e_1+f_1$, and $y=\lambda f_1 + e_2 + f_2$ be generators for a pair of $1$-spaces $(\braket{x},\braket{y})$ with $Q(x)=Q(y)=1$, and $b(x,y)=\lambda$. We count the number of $\braket{z}$ with $Q(z)=1$ such that $(\braket{x},\braket{z})\in\Delta_{\gamma_1}$ and $(\braket{z},\braket{y})\in \Delta_{\gamma_2}$; this computes the intersection number $p_{\gamma_1\gamma_2}^\lambda$. To do this, fix the signs $\epsilon_1$, and $\epsilon_2$ in $\{\pm 1\}$;  by the choice of coset representatives above, we count vectors $z$, with $Q(z)=1$, such that $b(x,z)=\epsilon_1\gamma_1$, and $b(z,y)=\epsilon_2\gamma_2$. By letting $z=\sum_i\, (\alpha_ie_i+\beta_if_i)$, one sees that $\alpha_1+\beta_1=\epsilon_1\gamma_1$, and hence there are $q$ choices for the pair $(\alpha_1,\beta_1)$. In the same way, there are $q$ pairs $(\alpha_2,\beta_2)$ satisfying $\alpha_2+\beta_2 = \epsilon_2\gamma_2 - \alpha_1\lambda$. Finally, we impose the condition $Q(z)=1$, which is equivalent to
    \begin{equation}\tag{\thesubsection a}\label{orth:eq:tailcount}
        \bsum_{i=3}^m\, \alpha_i\beta_i = 1-\alpha_1\beta_1-\alpha_2\beta_2.
    \end{equation}
    The number of vectors $v = \sum_{i=3}^m\, (\alpha_ie_i+\beta_i f_i)\in\F_q^{2m-4}$ satisfying this equation depends on whether $1-\alpha_1\beta_1-\alpha_2\beta_2$ is zero, or non-zero; denote these numbers by $A_\tau$, and $B_\tau$ respectively (by Lemma~\ref{lem:orth-norm-counts}). Hence, one must distinguish between these two cases; we count the number of times each case happens.
    For each $\gamma\in \F_q$, let $X_\gamma := \{(x,y)\in\F_{q}^2 : x+y=\gamma\}$, and $\phi_\gamma : X_\gamma\to \F_q$, be defined by sending $(x,y)\mapsto xy$, and let $\chi$ be the quadratic character, which sends $x\in \F_q^*$ to $1$ if $x$ is a square, to $-1$ if $x$ is a non-square, and sends $0$ to $0$. Also, we define the function $m(\alpha_1):=1-\epsilon_1\gamma_1\alpha_1+\alpha_1^2$. We count the number $r(\alpha_1)$, of solutions $(\alpha_2,\beta_2)\in \phi_{\epsilon_{2}\gamma_2-\lambda\alpha_1}^{-1}(m(\alpha_1))$; this is the number of roots of $P(T):=T^2-(\epsilon_2\gamma_2-\lambda\alpha_1)T+m(\alpha_1)$, and for each fixed choice of $(\alpha_1,\beta_1)$, determines how many pairs $(\alpha_2,\beta_2)$ force the right hand side of \eqref{orth:eq:tailcount} to be zero. Let $\Delta(\alpha_1)$ be the discriminant of $P(T)$ in $T$, regarded as a polynomial in $\alpha_1$, so that its discriminant $D(\Delta(\alpha_1))$ is independent of $\alpha_1$. Thus, one may write $\Delta_{\epsilon_1,\epsilon_2} = \Delta_{\epsilon_1,\epsilon_2}(X)\in \F_q[X]$ with $\Delta = AX^2+B_{\epsilon_1,\epsilon_2}X+C$, and
    \begin{align*}
        \begin{matrix}
            A = \lambda^2-4, & B_{\epsilon_1,\epsilon_2} = 2(2\epsilon_1\gamma_1 - \epsilon_2\gamma_2\lambda), & C = \gamma_2^2 - 4.
        \end{matrix}
    \end{align*}
    Applying the involution
    $(\epsilon_1,\epsilon_2)\mapsto (-\epsilon_1,-\epsilon_2)$ to the four possible sign configurations
    sends $B_{\epsilon_1,\epsilon_2}$ to $-B_{\epsilon_1,\epsilon_2}$. Consequently, 
    $$
        \Delta_{-\epsilon_{1},-\epsilon_{2}}(X) = \Delta_{\epsilon_{1},\epsilon_{2}}(-X).
    $$
    The substitution $X\mapsto -X$ is an involution of $\mathbb{F}_{q}$, so the sign configurations $(\epsilon_{1},\epsilon_{2})$ and $(-\epsilon_{1},-\epsilon_{2})$ contribute equally. Similarly, $(\epsilon_{1},-\epsilon_{2})$ and $(-\epsilon_{1},\epsilon_{2})$ also contribute equally. Thus it suffices to take $\epsilon_{1}=1$ and $\epsilon=\epsilon_{1}\epsilon_{2}\in \lbrace \pm 1 \rbrace$. Since each one-space $\left\langle z \right\rangle$ is represented by the two vectors $\pm z$, this common multiplicity is cancelled out when passing from vectors to one-spaces.
    
    Now simply write $\Delta_\epsilon = AX^2+B_\epsilon X+C$ with $A$, $B_\epsilon$, and $C$ as in Definition~\ref{orth:def:coeffs_and_tails}. By summing over $\alpha_1$, we see that the number of quadruples $(\alpha_1,\beta_1,\alpha_2,\beta_2)$ simultaneously satisfying $\alpha_1+\beta_1 = \gamma_1$, $\alpha_2 + \beta_2 = \epsilon\gamma_2 - \lambda\alpha_1$, and $1-\alpha_1\beta_1-\alpha_2\beta_2 = 0$ is
    \[
        \bsum_{\alpha_1\in \F_q}r(\alpha_1)
        =
        \bsum_{\alpha_1\in \F_q}\left(1+\chi(\Delta_\epsilon(\alpha_1))\right).
    \]
    Applying Lemma~\ref{orth:lem:quadratic-character-sum} gives precisely the value $c_{\gamma_1\gamma_2}^\lambda(\epsilon)$ in Definition~\ref{orth:def:coeffs_and_tails}. It follows that $q^2-c_{\gamma_1\gamma_2}^\lambda(\epsilon)$ is the non-vanishing tail coefficient, that is the number of remaining quadruples satisfying the desired conditions which cause the $(2m-4)$-dimensional tail sum in \eqref{orth:eq:tailcount} to be nonzero.
    
    If $\gamma_1=0$, then $\alpha_1+\beta_1=\epsilon_1\cdot 0=0$, and so the signatures $(\epsilon,\pm\epsilon)\sim (-\epsilon,\pm \epsilon)$ are identified, requiring this number to be divided by 2. The same extra factor of $2$ occurs if $\gamma_2=0$. If $\gamma_1=\gamma_2=0$, then all four signatures are identified, and the above needs to be divided by $4$. So let $Z$ denote the number of zeroes appearing in the pair $(\gamma_1,\gamma_2)$ and we divide the current running count by $2^{Z}$, giving the total running count as
    $$
        \frac{(c_{\gamma_1\gamma_2}^\lambda (1) + c_{\gamma_1\gamma_2}^\lambda (-1))A_\tau + (2q^2 - c_{\gamma_1\gamma_2}^\lambda (1) - c_{\gamma_1\gamma_2}^\lambda (-1))B_\tau}{2^Z}.
    $$

    It remains only to exclude the one-spaces $\left\langle x \right\rangle$ and $\left\langle y \right\rangle$. The first is counted precisely when $\gamma_{1}^{2}=4$ and $\gamma_{2}^{2}=\lambda^{2}$, and the second is counted precisely when $\gamma_{2}^{2}=4$ and $\gamma_{1}^{2}=\lambda^{2}$. Therefore, these endpoints are accounted for by subtracting the indicator sum 
    $$
        \mathbf{I}_{\gamma_{1}\gamma_{2}}^{\lambda} = \mathbf{1}(\gamma_{1}^{2}=4, \gamma_{2}^{2}=\lambda^{2}) + \mathbf{1}(\gamma_{2}^{2}=4,\gamma_{1}^{2}=\lambda^{2}),
    $$
    which is exactly what we wanted to prove.

    If instead $Q$ is a nondegenerate quadratic form of minus type, then the same argument runs through. The number of quadruples $(\alpha_1,\beta_1,\alpha_2,\beta_2)$ satisfying $\alpha_1+\beta_1=\gamma_1$ and $\alpha_2+\beta_2=\epsilon\gamma_2-\lambda\alpha_1$ is independent of the orthogonal type of $V$, so the vanishing tail coefficient is invariant. By choosing a standard minus type basis 
    \[
        \{e_1,f_1,\ldots, e_{m-1},f_{m-1},u,v\}
    \]
    with $b(e_i,f_j)=\delta_{ij}$, $b(e_i,e_j) = b(f_i,f_j)=Q(e_i)=Q(f_i)=0$, and $(Q(u), Q(v), b(u,v)) = (1,\zeta, 1)$, where $t^2+t+\zeta$ is irreducible over $\F_q$, one keeps the dimension the same. Thus the count differs only on the number of solutions $v$ in a $2(m-2)$-dimensional space $K$ satisfying $Q(v)=0$, where
    \[
        Q(v) = a^2 + ab + b^2\zeta + \bsum_{j=3}^{m-1}\, \alpha_j\beta_j.
    \]
    The number of such $v$ is given by $A_\tau$, and the number of $v$ satisfying $Q(v)=1$ is given by $B_\tau$ with $\tau=-1$.
\end{proof}

\subsection{Even characteristic}
Now let $\dim V=2m\ge6$ and $q=2^e$. We use the relations $\Delta_\lambda$ with $\lambda\in\F_q$ defined for even $q$ in Section~\ref{orth:subsec:common-setup}. The quadratic character criterion is replaced by the trace criterion in Lemma~\ref{orth:lem:artin-schreier}.

\begin{orthlemma}\label{orth:lem:artin-schreier}
    Let $K=\F_{2^e}$ and let $y\in K$.  The polynomial $X^2+X+y$ has two distinct roots in $K$ if and only if $\mathrm{Tr}_{K/\F_2}(y)=0$, and has no roots otherwise.
\end{orthlemma}

\begin{orthlemma}
    Let $K=\F_{2^e}$ and let $f(X)=aX^2+bX+c\in K[X]$. For every $x\in K$, one has
    \[
        \mathrm{Tr}_{K/\F_2}(f(x))
        =
        \mathrm{Tr}_{K/\F_2}(dx+c),
        \qquad
        d=a^{2^{e-1}}+b.
    \]
\end{orthlemma}
\begin{proof}
    Fix $x\in K$ and put $u=a^{2^{e-1}}x\in K$. Since $a^{2^e}=a$, one has $u^2=ax^2$. Also,
    \[
        \mathrm{Tr}_{K/\F_2}(u^2)
        =\bsum_{i=1}^{e}u^{2^i}
        =\bsum_{i=0}^{e-1}u^{2^i}
        =\mathrm{Tr}_{K/\F_2}(u),
    \]
    since $u^{2^e}=u$. Thus $ax^2$ and $a^{2^{e-1}}x$ have the same trace. Adding the trace of $bx+c$ to both sides gives
    \[
        \mathrm{Tr}_{K/\F_2}(ax^2+bx+c)
        =\mathrm{Tr}_{K/\F_2}\bigl((a^{2^{e-1}}+b)x+c\bigr),
    \]
    and the lemma follows.
\end{proof}

\begin{orthdefi}
    Let $q=2^e$ and let $\gamma_1,\gamma_2,\lambda\in\F_q$.  Set
    \[
        m_{\gamma_1}(x)=x^2+\gamma_1x+1,
        \qquad
        k_{\gamma_2}^{\lambda}(x)=\gamma_2+\lambda x.
    \]
    We define the \emph{vanishing tail coefficient}
    \[
        c_{\gamma_1,\gamma_2}^{\lambda}
        =
        \bsum_{x\in\F_q}r_{\gamma_1,\gamma_2}^{\lambda}(x),
    \]
    where $r_{\gamma_1,\gamma_2}^{\lambda}(x)$ is the number of distinct roots in $\F_q$ of the polynomial $T^2+k_{\gamma_2}^{\lambda}(x)T+m_{\gamma_1}(x)=0$.
\end{orthdefi}

\begin{orthprop}\label{orth:prop:even_char_general}
    Let $(V,Q)$ have dimension $2m\ge6$ and type $\tau\in\{\pm1\}$ over $\F_q$, where $q=2^e$, and let $\gamma_1,\gamma_2,\lambda\in\F_q$.  Define
    \[
        \mathrm I_{\gamma_1\gamma_2}^{\lambda}
        =
        \mathbf{1}(\gamma_1=0,\gamma_2=\lambda)
        +\mathbf{1}(\gamma_2=0,\gamma_1=\lambda).
    \]
    Then
    \[
        p_{\gamma_1\gamma_2}^{\lambda}
        =
        c_{\gamma_1,\gamma_2}^{\lambda}A_\tau
        +\left(q^2-c_{\gamma_1,\gamma_2}^{\lambda}\right)B_\tau
        -\mathrm I_{\gamma_1\gamma_2}^{\lambda}.
    \]
\end{orthprop}
\begin{proof}
    The same as Proposition~\ref{orth:prop:even_dim_orthog_pijk} except now there is no sign identification. 
\end{proof}

The vanishing tail coefficient in Proposition~\ref{orth:prop:even_char_general} is made explicit by the following lemma.  Write $\sigma(a)=a^{2^{e-1}}$, so that $\sigma(a)^2=a$.

\begin{orthlemma}
    The vanishing tail coefficient is given by the following cases.
    \begin{enumerate}[(i)]
        \item Suppose that $\lambda\gamma_2\ne0$.  Define
        \[
            a=\frac{\gamma_2^2+\gamma_1\gamma_2\lambda+\lambda^2}{\gamma_2^2},
            \qquad
            b=\frac{\gamma_1}{\gamma_2},
            \qquad
            c_0=\frac1{\gamma_2^2},
            \qquad
            d=\sigma(a)+b,
        \]
        and let
        \[
            \delta=\mathbf{1}\left(\mathrm{Tr}_{\F_q/\F_2}(d/\lambda+c_0)=0\right).
        \]
        Then
        \[
            c_{\gamma_1,\gamma_2}^{\lambda}
            =
            \begin{cases}
                2q-1,&d=0,\ \mathrm{Tr}_{\F_q/\F_2}(c_0)=0,\\
                1,&d=0,\ \mathrm{Tr}_{\F_q/\F_2}(c_0)=1,\\
                q+1-2\delta,&d\ne0.
            \end{cases}
        \]

        \item Suppose that $\lambda\ne0$ and $\gamma_1=\gamma_2=0$.  If
        \[
            \delta=\mathbf{1}\left(\mathrm{Tr}_{\F_q/\F_2}(\lambda^{-2})=0\right),
        \]
        then
        \[
            c_{0,0}^{\lambda}=q+1-2\delta.
        \]

        \item Suppose that $\lambda=0$ and $\gamma_2\ne0$.  Set
        \[
            r=\gamma_2^{-2},
            \qquad
            \beta=\sigma(r)+\gamma_1r.
        \]
        Then
        \[
            c_{\gamma_1,\gamma_2}^{0}
            =
            \begin{cases}
                2q,&\beta=0,\ \mathrm{Tr}_{\F_q/\F_2}(r)=0,\\
                0,&\beta=0,\ \mathrm{Tr}_{\F_q/\F_2}(r)=1,\\
                q,&\beta\ne0.
            \end{cases}
        \]

        \item If $\lambda=\gamma_2=0$, then
        \[
            c_{\gamma_1,0}^{0}=q.
        \]
    \end{enumerate}
    If $\lambda\ne0$, $\gamma_2=0$, and $\gamma_1\ne0$, the remaining case follows from the symmetry
    \[
        c_{\gamma_1,0}^{\lambda}=c_{0,\gamma_1}^{\lambda}.
    \]
\end{orthlemma}
\begin{proof}
    If $k_{\gamma_2}^{\lambda}(x)=0$, then the polynomial $P(T)=T^2+k_{\gamma_{2}}^{\lambda}(x)T+m_{\gamma_{1}}(x)$ has exactly one root.  Otherwise, substituting $T=k_{\gamma_2}^{\lambda}(x)t$ and applying Lemma~\ref{orth:lem:artin-schreier} gives
    \[
        r_{\gamma_1,\gamma_2}^{\lambda}(x)
        =
        \begin{cases}
            2,&\mathrm{Tr}_{\F_q/\F_2}\left(m_{\gamma_1}(x)/k_{\gamma_2}^{\lambda}(x)^2\right)=0,\\
            0,&\mathrm{Tr}_{\F_q/\F_2}\left(m_{\gamma_1}(x)/k_{\gamma_2}^{\lambda}(x)^2\right)=1.
        \end{cases}
    \]
    We proceed case by case:
    \begin{enumerate}[(i)]
        \item The change of variables
    \[
        u=\frac{x}{\lambda x+\gamma_2}
    \]
    is a bijection from $\F_q\setminus\{\gamma_2/\lambda\}$ to $\F_q\setminus\{1/\lambda\}$, and
    \[
        y:=\frac{m_{\gamma_1}(x)}{k_{\gamma_2}^{\lambda}(x)^2}
        =au^2+bu+c_0.
    \]
    The trace of $y$ is $\mathrm{Tr}_{\F_q/\F_2}(du+c_0)$. 
    
    If $d=0$, then either $\operatorname{Tr}(c_{0})=0$ and all $q-1$ values of $u$ are allowed and each such choice contributes $2$ to the count, or $\operatorname{Tr}(c_{0})\ne 0$ and no choices of $u$ give $\operatorname{Tr}(y)=0$. The excluded value $x=\gamma_2/\lambda$, for which $k_{\gamma_2}^{\lambda}(x)=0$, contributes one root separately. Hence $c_{\gamma_{1},\gamma_{2}}^{\lambda}=2q-1$ when $\operatorname{Tr}(c_0)=0$, and $1$ otherwise. 
    
    If $d\ne0$, there are $q/2$ solutions on $\F_q$. Of these, the omitted value $u=1/\lambda$ is removed precisely when counted by $\delta$. The excluded value $x=\gamma_2/\lambda$, for which $k_{\gamma_2}^{\lambda}(x)=0$, contributes one root separately, giving (i).

    \item Since $\gamma_{1}=\gamma_{2}=0$ and $\lambda\ne 0$, we have that $m_{0}(x)=x^{2}+1$ and $k_{0}^{\lambda}(x)=\lambda x$. The value $x=0$ contributes one root. For $x\ne0$, put $u=1/x$; then 
    $$
        y=\frac{x^{2}+1}{\lambda^{2}x^{2}}=\lambda^{-2}u^{2}+\lambda^{-2}.
    $$
    This reduces the trace condition to
    \[
        \mathrm{Tr}_{\F_q/\F_2}\left(\lambda^{-2}+\lambda^{-2}u^2\right)=0.
    \]
    There are $q/2$ solutions in $u\in\F_q$ satisfying this trace equation and the omitted value $u=0$ is one of them precisely when $\delta=1$. It follows that this case contributes $1+2(q/2-\delta)$.

    \item We have that $\lambda=0$ and $\gamma_{2}\ne 0$, so 
    $$
        y=\gamma_{2}^{-2}x^{2}+\gamma_{1}\gamma_{2}^{-2}x+\gamma_{2}^{-2}.
    $$
    Letting $r=\gamma_{2}^{-2}$ and $\beta=\sigma(r)+\gamma_{1}r$, the trace condition becomes
    \[
        \mathrm{Tr}_{\F_q/\F_2}(\beta x+r)=0.
    \]
    It has either all, none, or half of the field as solutions according to the three displayed cases, and every solution supplies two roots. 
    
    \item Here $k_0^0(x)=0$ for every $x$. Hence the polynomial becomes $T^2+m_{\gamma_1}(x)$, which has exactly one root in $\F_q$ for each $x$. Thus $c_{\gamma_1,0}^{0}=q$.
    \end{enumerate}
    The final symmetry follows by interchanging $x$ and $T$ in the polynomial $T^{2}+k_{0}^{\lambda}(x)T+m_{\gamma_{1}}(x)$. 
\end{proof}

Suppose that a larger group fuses the non-diagonal orbitals. For each fused orbital $R_i$, let $I_i$ be the set of original indices whose relations form $R_i$, so that $R_i=\bigcup_{\gamma\in I_i}\Delta_\gamma$. For each $\lambda\in I_k$, the intersection numbers of the fused action are given by
\[
    p'{}_{ij}^{k}
    =\sum_{\gamma_1\in I_i}\sum_{\gamma_2\in I_j}
      p_{\gamma_1\gamma_2}^{\lambda}.
\]
The sum is independent of the choice of $\lambda\in I_k$, and the diagonal remains separate.

\subsection{Odd dimension in odd characteristic}
Returning to odd $q$, we retain the notation from Proposition~\ref{orth:prop:even_dim_orthog_pijk}.  In particular, the functions $c_{\gamma_1,\gamma_2}^{\lambda}$ and $\rho_{\gamma_1,\gamma_2}^{\lambda}$ are defined by the same polynomials $\Delta_\epsilon$, and we use the quadratic character sum in Lemma~\ref{orth:lem:quadratic-character-sum}. Let $\dim V = 2m+1\ge5$, so that $m\ge2$. One may write 
$$
    V = H_{1}\perp H_{2}\perp \dots \perp H_{m}\perp \left\langle t \right\rangle,
$$
where for each $i$, $H_{i}$ is a hyperbolic plane $\left\langle e_{i},f_{i} \right\rangle$ with $Q(e_i)=Q(f_i)=0$ and $b(e_i,f_i)=1$, and $t$ is a basis vector with $Q(t)=d$. Further, let $\eta=\chi(d)\in\{\pm1\}$.

Write $z=\xi t+\bsum_{i=1}^m(\alpha_i e_i+\beta_i f_i)$.
We write $p_{\gamma_1\gamma_2}^{\lambda,\eta}$ for the intersection numbers with orthogonal type parameter $\eta=\chi(d)$.
When $m=2$, the tail space is zero-dimensional, and we use $A_1=1$ and $B_1=0$.

\begin{orthprop}\label{orth:prop:odd_dim_orthogonal_pijk}
    Let $(V,Q)$ and $\eta=\chi(d)$ be as above, where $q$ is odd, and let $\gamma_1,\gamma_2,\lambda\in\Lambda$. Define $\mathrm I_{\gamma_1\gamma_2}^{\lambda}$ and $Z$ as in Proposition~\ref{orth:prop:even_dim_orthog_pijk}.  Then
    \[
        p_{\gamma_1\gamma_2}^{\lambda,\eta}
        =
        \frac{
        2q^{2m-2}
        +\eta q^{m-1}
        \left(
        \rho_{\gamma_1,\gamma_2}^{\lambda}(1)
        +\rho_{\gamma_1,\gamma_2}^{\lambda}(-1)-2
        \right)
        }{2^Z}
        -\mathrm I_{\gamma_1\gamma_2}^{\lambda}.
    \]
\end{orthprop}
\begin{proof}
    The orbitals $\Delta_{\lambda}$ are represented exactly as in Proposition~\ref{orth:prop:even_dim_orthog_pijk} with $\lambda\in \Lambda$. 

    Let $x=e_1+f_1$ and $y=\lambda f_1+e_2+f_2$.  We count the one-spaces $\braket{z}$ with $Q(z)=1$ such that $(\braket{x},\braket{z})\in\Delta_{\gamma_1}$ and $(\braket{z},\braket{y})\in\Delta_{\gamma_2}$.  Fix signs $\epsilon_1,\epsilon_2\in\{\pm1\}$.  Since $t$ is orthogonal to both $x$ and $y$, the conditions $b(x,z)=\epsilon_{1}\gamma_{1}$, and $b(z,y)=\epsilon_{2}\gamma_{2}$ give
    \[
        \alpha_1+\beta_1=\epsilon_1\gamma_1,
        \qquad
        \alpha_2+\beta_2=\epsilon_2\gamma_2-\lambda\alpha_1.
    \]
    Thus there are again $q$ choices for each of the pairs $(\alpha_1,\beta_1)$ and $(\alpha_2,\beta_2)$.

    Define $\phi$ as follows:
    \[
        \phi=1-\alpha_{1}\beta_{1}-\alpha_{2}\beta_{2}.
    \]
    The $Q(z)=1$ condition is equivalent to
    \begin{equation}\tag{\thesubsection b}\label{orth:eq:odd-tailcount}
        \bsum_{i=3}^{m}\alpha_i\beta_i
        =
        \phi-d\xi^2.
    \end{equation}
    The orthogonal sum $H_3\perp\cdots\perp H_m$ has dimension $2m-4$ and plus type, so there are $A_1$ possible tail vectors when the right hand side of~\eqref{orth:eq:odd-tailcount} vanishes and $B_1$ otherwise.

    As in the proof of Proposition~\ref{orth:prop:even_dim_orthog_pijk}, set $m(\alpha_1)=1-\epsilon_1\gamma_1\alpha_1+\alpha_1^2$.  For fixed $\alpha_1$, we have that $\phi=P(\alpha_2)$, where
    \[
        P(T)=T^2-(\epsilon_2\gamma_2-\lambda\alpha_1)T+m(\alpha_1).
    \]
    The discriminant of this polynomial is $\Delta_{\epsilon_1,\epsilon_2}(\alpha_1)$ and the involution
    $$
        (\epsilon_{1},\epsilon_{2}) \mapsto (-\epsilon_{1},-\epsilon_{2})
    $$
    on the four sign configurations therefore reduces the calculation to $\epsilon_1=1$ and $\epsilon=\epsilon_1\epsilon_2\in\{\pm1\}$, exactly as in Proposition~\ref{orth:prop:even_dim_orthog_pijk}.

    Fix one such value of $\epsilon$. By Lemma~\ref{orth:lem:quadratic-character-sum}, the coefficient from Definition~\ref{orth:def:coeffs_and_tails}
    \[
        c_{\gamma_1,\gamma_2}^{\lambda}(\epsilon)
        =\bsum_{\alpha_1\in\F_q}
        \left(1+\chi(\Delta_\epsilon(\alpha_1))\right)
    \]
    is the number of heads for which $\phi=0$.  Hence $q^2-c_{\gamma_1,\gamma_2}^{\lambda}(\epsilon)$ heads have $\phi\ne0$. Applying the monic case of the same lemma to $P(T)$ gives
    \[
        \bsum_{\alpha_2\in\F_q}\chi(\phi)
        =
        \begin{cases}
            q-1,&\Delta_\epsilon(\alpha_1)=0,\\
            -1,&\Delta_\epsilon(\alpha_1)\ne0.
        \end{cases}
    \]
    Subsequently, one has that
    \[
        \bsum_{\alpha_1,\alpha_2\in\F_q}\chi(\phi)
        =q\left(\rho_{\gamma_1,\gamma_2}^{\lambda}(\epsilon)-1\right).
    \]
    There are three distinct cases which determine the vanishing of the right hand side of \eqref{orth:eq:odd-tailcount}:
    \begin{enumerate}[(i)]
        \item $\phi=0$ and there is one possible value of $\xi$, namely $\xi=0$;
        \item $\phi\ne 0$ and $\chi(\phi)=1$, so there exist $1+\eta$ values of $\xi$ making $\phi-d\xi^{2}=0$; 
        \item $\phi\ne 0$ and $\chi(\phi)=-1$, so there exist $1-\eta$ values of $\xi$ making $\phi-d\xi^{2}=0$.
    \end{enumerate}
    Let the number of heads which place $\phi$ into the respective cases of (i), (ii), and (iii) be $c$, $X$, and $Y$. It follows that $c=c_{\gamma_{1}\gamma_{2}}^{\lambda}(\epsilon)$ is the vanishing tail coefficient from Definition~\ref{orth:def:coeffs_and_tails}. The value $X-Y=\sum_{\alpha_1,\alpha_2\in\F_q}\chi(\phi)$ measures the difference between the number of times $\phi$ is a square and nonsquare, hence
    $$
        X-Y=q(\rho_{\gamma_{1}\gamma_{2}}^{\lambda}(\epsilon)-1), \qquad X+Y=q^{2}-c_{\gamma_{1}\gamma_{2}}^{\lambda}(\epsilon).
    $$ 
    Hence, for an arbitrary value of $\eta$, the number $N_{0}(\epsilon)$ of pairs consisting of a head and a value of $\xi$ for which $\phi-d\xi^2=0$ is given by 
    $$
        N_{0}(\epsilon)=(1+\eta)X + (1-\eta)Y+c.
    $$
    Since there are $q^2$ possible heads and $q$ choices of $\xi$, there are $q^{3}$ pairs consisting of a head and a value of $\xi$; after summing over the two choices of $\epsilon$ the total running count is given by 
    $$
        (N_{0}(1)+N_{0}(-1))A_{1}+(2q^{3}-N_{0}(1)-N_{0}(-1))B_{1}.
    $$
    If $Z$ counts the number of $\gamma_1$ or $\gamma_2$ equal to zero, the same sign identifications as in the proof of Proposition~\ref{orth:prop:even_dim_orthog_pijk} occur and we divide by $2^Z$. Finally, the one-spaces $\braket{x}$ and $\braket{y}$ both have zero anisotropic coordinate.  They occur under precisely the same conditions as before, and so are removed by subtracting $\mathrm I_{\gamma_1\gamma_2}^{\lambda}$.
\end{proof}

\par
\endgroup

\section{Proof of Theorem~\ref{thm:rank-four-srg}}\label{sec:main-proof}
Recall that $\mathcal{G}$ is the set of non-affine primitive permutation groups of rank four. By Theorem~\ref{thm:rank-four-groups}, the groups $G\in \mathcal{G}$ are either almost simple and belong to Tables~\ref{tab:rank-four-families} or~\ref{tab:rank-four-exceptions}, or are in conclusions (i), (ii), or (iii) of Theorem~\ref{thm:rank-four-groups}. For each such $G\in \mathcal{G}$, we inspect its orbital graphs and check for strong regularity. Since we classify the graphs up to complementation, we choose $\Gamma$ with edge relation a single non-diagonal self-paired orbital of $G$. Its complement is the union of the other two non-diagonal orbitals, so it suffices to check these single-orbital representatives for strong regularity. As each of the groups in Table~\ref{tab:rank-four-exceptions} can be dealt with computationally as described in Section~\ref{sec:computational-tools}, we need only consider the cases where $G$ is an almost simple group belonging to Table~\ref{tab:rank-four-families}, or is not almost simple, and belongs to one of conclusions (i), (ii), or (iii) of Theorem~\ref{thm:rank-four-groups}.

\subsection{The non-almost simple automorphism group case}

Suppose that $G\in \mathcal{G}$, and $G$ is not almost simple. By Theorem~\ref{thm:rank-four-groups}, one of the following holds:
\begin{enumerate}
    \item The group $G$ is of simple diagonal type, and $S=\soc{G}\cong A_5\times A_5$.
    \item The group $G\cong \PSL_2(8)^2.6$ is such that $T^2\le G \le K\wr S_2$ is in its product action of degree $28^2$, where $T\cong \PSL_2(8)$, and $K\cong \PSL_2(8).3$.
    \item The group $G$ with socle $T^3$, is of product type on $\Omega=\Delta^3$, and $T^3\le G\le K\wr S_3$, where $K$ is an almost simple group, acting $2$-transitively on $\Delta$.
\end{enumerate}
In the first two cases, we use the GAP computations described in Section~\ref{sec:computational-tools} to see that strongly regular orbital graphs occur only in Case $2$, whose three non-diagonal self-paired orbitals give precisely three strongly regular graphs, with parameters $(784,54,26,2)$, $(784,243,82,72)$ and $(784,486,298,306)$. In Case $3$, since $K$ is $2$-transitive on $\Delta$, the orbitals of $K\wr S_3$ are the same as those of $S_d\wr S_3$, where $d := |\Delta|$, and these are the generalised Hamming graphs, $\Gamma_1$, $\Gamma_2$, and $\Gamma_3$, where $\Gamma_i$ is the graph on $\Delta^3$, and vertices are joined by an edge if and only if they disagree in $i$ coordinates. Indeed, for $\alpha=(\delta,\delta,\delta)$, the group $K_\delta^3$ can move each coordinate independently within $\Delta\setminus\{\delta\}$, while $S_3$ permutes the coordinates, so the orbits of the stabiliser are determined by the number of coordinates differing from $\alpha$; since $G\le K\wr S_3$ and both groups have rank four, their orbitals coincide. Let $M= (m_{ij})$ be the $3\times 3$ matrix with $m_{ij} = p^i_{jj}$. Fix two vertices differing in $i$ coordinates; where they agree, a third vertex can agree with both or use one of the $d-1$ other values; where they differ, it can agree with either of the two fixed coordinates or use one of the $d-2$ values differing from both. Counting the choices that give Hamming distance $j$ from each fixed vertex yields $p_{jj}^i$. Thus, by basic counting, we have
\[
    M = \begin{pmatrix}
        * & 2(d-1)(d-2) & (d-2)(d-1)^2 \\
        2 & * & (d-1)(d-2)^2 \\
        0 & 6(d-2) & *
    \end{pmatrix}.
\]
The stars $\ast$ suppress the intersection numbers $p_{jj}^j$. For the graph $\Gamma_j$, strong regularity is equivalent to equality of the two unstarred entries in column $j$, corresponding to the other two non-diagonal relations; no comparison between different columns is needed. The graph $\Gamma_2$ is strongly regular if and only if $d=4$, which is impossible since a faithful permutation group $K$ of degree at most four embeds in the soluble group $S_4$, whereas an almost simple group has a nonabelian simple socle.

\subsection{The graphs whose automorphism group belongs to Table~\ref{tab:rank-four-families}}

By Theorem~\ref{thm:rank-four-groups}, one can divide the actions in Table~\ref{tab:rank-four-families} into the following three categories: alternating socle; parabolic actions of groups of Lie type; nonsingular subspace actions of classical groups. We treat these one at a time.

\subsubsection{Alternating socle}
\begin{lemma}
    Suppose that $G$ is in Table~\ref{tab:rank-four-families} and has socle $A_n$. Then the orbital graphs of $G$ are generalised Johnson graphs, and there are two strongly regular graphs when $n\in\{7,10\}$ with parameters $(35,18,9,9)$, and $(120,63,30,36)$ respectively.
\end{lemma}
\begin{proof}
    Since the socle is $A_n$ with $n\ge 7$, the group $G$ is either $A_n$ or $S_n$, and hence, for $0\le i\le3$, the orbitals are described by
    \[
        \Delta_i = \{(x,y)\in \Omega_3\times \Omega_3 : |x\cap y|=3-i\}.
    \]
    The graphs of these correspond to the generalised Johnson graphs, and by \cite{praeger}, the only strongly regular ones occur when $i=2$, $n\in \{7,10\}$, and have parameters 
    $(35,18,9,9)$, and $(120,63,30,36)$.
\end{proof}

\subsubsection{Parabolic actions}

In this section, we treat the parabolic actions of rank four. Associated to many of these actions is a distance regular graph, which we use to compute specific intersection numbers. The dual polar and halved dual polar arrays are given in \cite[Theorems 9.4.3 and 9.4.8]{drg}, and the exceptional Lie type arrays in \cite[Table 10.8]{drg}. We list the relevant intersection arrays below, using the following notation. The Gaussian coefficient $\qbinom{n}{k}_q$ counts the $k$-dimensional subspaces of an $n$-dimensional vector space over $\F_q$; when the subscript is omitted, the field size is understood to be $q$. For brevity, we define
\[
    [j]_q=\qbinom{j}{1}_q=\frac{q^j-1}{q-1},
\]
and write $\Phi_j=\Phi_j(q)$ for the $j$th cyclotomic polynomial. Tables~\ref{tab:parabolic-arrays} and~\ref{tab:hexagon-intersections} record the relevant intersection data used in the proofs below and $K$ denotes the indicated acting group and $K_\alpha$ a point stabiliser in that group; $K$ need not be simple. The parameters $q,r$ are prime powers, with $q$ odd for $\PSp_6(q)$, $q\ge2$ for $G_2(q)$ and $n\ge6$ for $\PSL_n(q)$. The unitary arrays for $\U_6(r)$ and $\U_7(r)$ use the underlying field of order $r^2$.

\begin{table}[H]
    \centering
    \small
    \renewcommand{\arraystretch}{1.35}
    \begin{tabular}{lll}
    \hline
    $K$ & $K_\alpha$ & Intersection array $\{b_0,b_1,b_2;c_1,c_2,c_3\}$ \\
    \hline
    $\PSp_6(q)$ & $P_3$ & $\{q[3]_q,q^2[2]_q,q^3;1,[2]_q,[3]_q\}$ \\
    $\Pommega_7(q)$ & $P_3$ & $\{q[3]_q,q^2[2]_q,q^3;1,[2]_q,[3]_q\}$ \\
    $\U_6(r)$ & $P_3$ & $\{r[3]_{r^2},r^3[2]_{r^2},r^5;1,[2]_{r^2},[3]_{r^2}\}$ \\
    $\U_7(r)$ & $P_3$ & $\{r^3[3]_{r^2},r^5[2]_{r^2},r^7;1,[2]_{r^2},[3]_{r^2}\}$ \\
    $\Pommega_8^-(q)$ & $P_3$ & $\{q^2[3]_q,q^3[2]_q,q^4;1,[2]_q,[3]_q\}$ \\
    $G_2(q)$ & $P_1,P_2$ & $\{q(q+1),q^2,q^2;1,1,q+1\}$ \\
    ${}^3D_4(q)$ & $P_1$ & $\{q^4+q,q^4,q^4;1,1,q^3+1\}$ \\
    ${}^3D_4(q)$ & $P_2$ & $\{q^3(q+1),q^4,q^4;1,1,q+1\}$ \\
    $\PSL_n(q)$ & $P_3$ & $\{q[3]_q[n-3]_q,q^3[2]_q[n-4]_q,q^5[n-5]_q;
        1,[2]_q^2,[3]_q^2\}$ \\
    $\Pommega_{12}^+(q)$ & $P_5,P_6$ & $\{q\Phi_6\Phi_3\Phi_5,q^5\Phi_4\Phi_3,q^9;
        1,\Phi_4\Phi_3,\Phi_6\Phi_3\Phi_5\}$ \\
    $\Pommega_{14}^+(q)$ & $P_6,P_7$ & $\{q\Phi_6\Phi_3\Phi_7,q^5\Phi_4\Phi_5,q^9\Phi_3;
        1,\Phi_4\Phi_3,\Phi_6\Phi_3\Phi_5\}$ \\
    $E_7(q)$ & $P_7$ & $\{q\Phi_6\Phi_3^2\Phi_{12}\Phi_9,q^9\Phi_8\Phi_5,q^{17};1,\Phi_8\Phi_5,\Phi_6\Phi_3^2\Phi_{12}\Phi_9\}$ \\
    \hline
    \end{tabular}
    \caption{Intersection arrays of the distance regular orbital graphs for the parabolic actions in Lemma~\ref{lem:parabolic-srg}.}
    \label{tab:parabolic-arrays}
\end{table}

\begin{table}[H]
    \centering
    \small
    \renewcommand{\arraystretch}{1.5}
    \begin{tabular}{llcccc}
    \hline
    $K$ & $K_\alpha$ & $p_{33}^1$ & $p_{33}^2$ & $p_{22}^1$ & $p_{22}^3$ \\
    \hline
    $G_2(q)$ & $P_1,P_2$ & $q^4(q-1)$ & $q^4(q-1)$ & $q^2(q-1)$ & $(q-1)(q+1)^2$ \\
    ${}^3D_4(q)$ & $P_1$ & $q^8(q-1)$ & $q^5(q-1)(q^3+q+1)$ & $q^4(q-1)$ & $(q-1)(q^3+1)^2$ \\
    ${}^3D_4(q)$ & $P_2$ & $q^8(q^3-1)$ & $q^5(q^6-q^3-q^2+1)$ & $q^4(q^3-1)$ & $(q^3-1)(q+1)^2$ \\
    $E_7(q)$ & $P_7$ & $q^{26}\Phi_1$ & $q^{21}\Phi_1^2\Phi_5$ &
        $q^9\Phi_8(q^5[8]_q+q^9-1)$ & $\Phi_1\Phi_3^2\Phi_5\Phi_6\Phi_8\Phi_9\Phi_{12}$ \\
    \hline
    \end{tabular}
    \caption{Intersection numbers for the parabolic actions of exceptional groups of Lie type.}
    \label{tab:hexagon-intersections}
\end{table}
\par\medskip

\begin{lemma}\label{lem:parabolic-srg}
    Suppose that $G\in \mathcal{G}$ is either an almost simple group in Table~\ref{tab:rank-four-families}, with socle a group of Lie type not isomorphic to $\PSL_3(q)$ and with point stabiliser a parabolic subgroup, or $G=G_2(2)$ in one of its two parabolic actions. Then the strongly regular orbital graphs $\Gamma$, for $G$ are as follows, where $S=\soc{G}$:
    \begin{enumerate}[(i)]
        \item $G\triangleright G_2(q)$ with $q\ne 3^a$, $H\cap G_2(q) = P_2$, and $\Gamma$ is the distance-$3$ graph of $G_{2,2}$.
        \item $G\triangleright G_2(q)$ with $H\cap G_2(q) = P_1$, and $\Gamma$ is the complement of the $\Omega_7(q)$ polar graph.
        \item $S\cong \Pommega_7(q)$, with $H\cap S =P_3$, and $\Gamma$ is the complement of the $\mathrm{O}_8^+(q)$ polar graph.
        \item $S \cong \PSp_6(q)$ with $q$ odd, $H\cap S = P_3$, and $\Gamma$ is the distance-$3$ graph of the dual polar graph for $\PSp_6(q)$.
    \end{enumerate}
\end{lemma}
\begin{proof}
    For the actions in Table~\ref{tab:rank-four-families}, write $S=\soc{G}$ and let $H\cap S$ be the relevant parabolic subgroup. In these actions, if $S$ is classical and not $\PSL_n(q)$ or $\Pommega_{2m}^+(q)$, then one of the orbital graphs of $G$ is a dual polar graph which is distance transitive and has intersection array as given in \cite[Theorem 9.4.3]{drg}; if $S\cong \PSL_n(q)$, then one of the orbital graphs is the Grassmann graph, which is distance transitive, with parameters given below; if $S$ is an exceptional group of Lie type, then one of the orbital graphs is distance transitive with intersection array given as in \cite[Table 10.8]{drg}; and if $S\cong \Pommega_{2m}^+(q)$, then one of the orbital graphs is a halved dual polar graph with intersection array as in \cite[Theorem 9.4.8]{drg}. These arrays, together with those of the two $G_2(2)$ actions, are listed in Table~\ref{tab:parabolic-arrays}. In each case, one of the orbital graphs, say $\Gamma_1$, is distance transitive with intersection array $\{b_0,b_1,b_2;c_1,c_2,c_3\} = \{p_{11}^0,p_{21}^1,p_{31}^2;p_{01}^1,p_{11}^2,p_{21}^3\}$. As in \cite[Lemma 4.1.7]{drg}, we have
    \begin{equation}\label{eq:parabolic-recurrence}
        p_{i+1,j}^k = \frac{1}{c_{i+1}}(p_{i,j-1}^kb_{j-1} + p_{ij}^k(a_j-a_i) + p_{i,j+1}^kc_{j+1}-p_{i-1,j}^kb_{i-1}).
    \end{equation}
    Call the two remaining non-trivial orbital graphs $\Gamma_2$ and $\Gamma_3$, corresponding to distances $2$ and $3$ in $\Gamma_1$. The graph $\Gamma_1$ has diameter $3$ and is not strongly regular. For $i\in \{2,3\}$, the graph $\Gamma_i$ is strongly regular if and only if $p_{ii}^r = p_{ii}^s$, where $\{i,r,s\}=\{1,2,3\}$. Given the intersection array, we compute the required intersection numbers using \eqref{eq:parabolic-recurrence} and Lemma~\ref{lem:example}, and check whether these equalities hold.
    
    Suppose that $G\triangleright G_2(q)$ with $H\cap G_2(q)=P_1$ or $P_2$, allowing $q=2$, or that $\soc{G}={}^3D_4(q)$ with $H\cap S=P_1$ or $P_2$. Their distance transitive orbital graphs have the intersection arrays in Table~\ref{tab:parabolic-arrays}. Applying the recursion \eqref{eq:parabolic-recurrence} gives the four intersection numbers in Table~\ref{tab:hexagon-intersections}. The difference $p_{33}^1-p_{33}^2$ is zero for both $G_2(q)$ actions, negative for the action of ${}^3D_4(q)$ on $P_1$-cosets and positive for the action of ${}^3D_4(q)$ on $P_2$-cosets. The difference $p_{22}^1-p_{22}^3$ is negative for both $G_2(q)$ actions and the action of ${}^3D_4(q)$ on $P_1$-cosets, and positive for the action of ${}^3D_4(q)$ on $P_2$-cosets, since $q^2>q+1$ for $q\ge2$. Thus neither ${}^3D_4(q)$ action gives a strongly regular orbital graph, whereas for $G_2(q)$ the graph $\Gamma_3$ is strongly regular with parameters $(\frac{q^6-1}{q-1},q^5,q^4(q-1),q^4(q-1))$. Since the action of $G_{2}(q)$ on $P_1$-cosets is contained in the action of $\Omega_7(q)$ on $P_1$-cosets, for the $G_2(q)/P_1$ action, we see that this strongly regular graph is the complement of the $\Omega_7(q)$ polar graph. For the $G_2(q)/P_2$ action, this strongly regular graph is the distance-$3$ graph of $G_{2,2}$, and when $q = 3^a$ this graph is isomorphic to the complement of the $\Omega_7(q)$ polar graph. For the action of $E_7(q)$ on $P_7$-cosets, the same recursion gives the counts in Table~\ref{tab:hexagon-intersections}, and neither equality holds.

    Suppose now that the socle of $G$ is one of $\PSp_6(q)$, $\U_n(r)$ with $n\in\{6,7\}$ and $r^2=q$, $\Pommega_7(q)$, or $\Pommega_8^-(q)$ so that $\Gamma_1$ is a dual polar graph. We compute
    \begin{align*}
        p_{22}^1 = q^{e+1}(q+1)(q^e-1), \quad &p_{22}^3 = \frac{q^2+q+1}{q+1}\left((q^2+q+1)(q^e-1) + (q+1)(q^{e+1}-1) -  q^{e+2} + 1\right) \\
        p_{33}^1 = q^{2e+3}(q^e-1), \quad &p_{33}^2 = \frac{q^{e+2}}{q+1}\left(q^e(q^2-1) + (q^e-1)(q^{e+2}+q^{e+1}-q^2-q)\right),
    \end{align*}
    where $e=1$ if $S = \PSp_6(q)$ or $S=\Pommega_7(q)$, and $e = \frac{1}{2},\frac{3}{2}$ if $S = \U_6(r),\U_7(r)$ respectively, where $q=r^2$, and $e=2$ if $S=\Pommega_8^-(q)$. Subtracting gives
    \begin{align}
        p_{22}^1-p_{22}^3&=(q+1)(q^e-1)(q^{e+1}-q^2-q-1),\label{eq:dual-polar-distance-two}\\
        p_{33}^1-p_{33}^2&=q^{e+2}(q^e-q).\label{eq:dual-polar-distance-three}
    \end{align}
    The first difference, \eqref{eq:dual-polar-distance-two}, is nonzero for each permitted value of $e$. The second, \eqref{eq:dual-polar-distance-three}, vanishes precisely when $e=1$. Thus neither the unitary actions nor the action of $\Pommega_8^-(q)$ gives a strongly regular orbital graph. If $e=1$, then $p_{33}^1=p_{33}^2$, and so $(\PSp_6(q),P_3)$, and $(\Pommega_7(q),P_3)$ both have a strongly regular orbital graph, the distance-$3$ graph of the respective dual polar graphs; when $q$ is even, these graphs are isomorphic. By observing that the action of $\Pommega_7(q)$ on totally singular $3$-spaces is contained in the action of $\Pommega_8^+(q)$ on singular $1$-spaces, we see that this graph is the complement of the $\mathrm{O}_8^+(q)$ polar graph. If $q$ is odd, then $\PSp_6(q)\not\cong \Pommega_7(q)$, and the distance-$3$ graph of the dual polar graph for $\PSp_6(q)$ is not isomorphic to the complement of the $\mathrm{O}_8^+(q)$ polar graph.
    
    Next let $S=\PSL_n(q)$ be as in Table~\ref{tab:rank-four-families}, so that $\Gamma_1$ is the distance transitive Grassmann graph with parameters
    \[
        b_i = q^{2i+1}\qbinom{3-i}{1}\qbinom{n-3-i}{1}, \qquad c_i = \qbinom{i}{1}^2
    \]
    for $0\le i\le3$. By \eqref{eq:parabolic-recurrence} and Lemma~\ref{lem:example}, after letting $X=q^n$, we obtain
    \begin{align*}
        f_2=p_{22}^1-p_{22}^3
        &=\frac{(X-q^5-q^4+q^2)(X-q^5-2q^4+q^2+q)}{q^3(q-1)^2},\\
        f_3=p_{33}^1-p_{33}^2
        &=\frac{(X-q^5)(X-q^5-q^4+q^3)}{q^2(q-1)^2(q+1)}.
    \end{align*}
    Since $n\ge6$ and $q\ge2$, we have
    \[
        X\ge q^6\ge q^5+2q^4,\qquad
        q^6-q^5-2q^4=q^4(q-2)(q+1)\ge0.
    \]
    Thus every numerator factor and both denominators are strictly positive, so $f_2,f_3>0$. Hence neither the distance-$2$ graph $\Gamma_2$ nor the distance-$3$ graph $\Gamma_3$ is strongly regular.
    
Finally suppose that $S\cong \Pommega_{2m}^+(q)$ with $m\in\{6,7\}$. The two actions are on the two families of maximal totally singular $m$-spaces. Their distance transitive orbital graphs are halved dual polar graphs whose intersection numbers agree, so it suffices to consider one family. As above, applying the recurrence gives $f_2=p_{22}^1-p_{22}^3$ and $f_3=p_{33}^1-p_{33}^2$, where, for $m=6$,
\begin{align*}
    f_2&=(q^2+1)^2(q^3-q-1)(q^5-q^4-q^2-1),\\
    f_3&=q^{11}(q-1),
\end{align*}
and, for $m=7$,
\begin{align*}
    f_2&=(q-1)(q^2+1)^2\Phi_5(q^7+q^6-q^4-q^2-q-1),\\
    f_3&=q^{11}(q^2+1)(q^3+q^2-1).
\end{align*}
Each expression is positive for $q\ge2$, so neither the distance-$2$ nor the distance-$3$ orbital graph is strongly regular.
\end{proof}

\begin{remark}
At $q=2$, the acting group $G_2(2)\cong\U_3(3):2$ is not simple and has socle $\U_3(3)$. Its $P_1$-coset and its $P_2$-coset actions have corresponding strongly regular distance-$3$ graphs with parameters $(63,32,16,16)$. The two graphs are not isomorphic and occur in Theorem~\ref{thm:rank-four-srg} (iv) and (vi), respectively.
\end{remark}

 \begin{prop}
    Let $G\in \mathcal{G}$ be the entry in Table~\ref{tab:rank-four-families} with $S=\PSL_3(q)$, $H\cap S = P_{1,2}$, and suborbits $\Delta_i$, for $0\le i\le 3$. Define $M=(m_{ij})$, to be the matrix with $m_{ij} = p_{jj}^i$. We have that
    \[
        M =\begin{pmatrix}
            * & q(q-1) & q^2(q-1) \\
            1 & * & q(q-1)^2\\
            0 & 4(q-1) & *
        \end{pmatrix}
    \]
    The stars represent omitted values $p_{jj}^j$. The suborbit lengths are $1$, $2q$, $2q^2$, and $q^3$. Consequently, the only strongly regular orbital graph is the graph of $\Delta_2$, with $q=4$ and parameters $(105,32,4,12)$.
\end{prop}
\begin{proof}
    The $G$-action is on pairs of subspaces $(U,W)$, with $U\subseteq W$, and $\dim(U)=1$, $\dim(W)=2$. Define $\alpha = (U,W)$ so that the orbits of $G_\alpha$ on the pairs $(U',W')$ described above are given by:
    \begin{align*}
        \Delta_0&=\{(U,W)\} \\
        \Delta_1&=\{(U',W') : U'=U \text{ or } W'=W\}\setminus \Delta_0 \\
        \Delta_2 &= \{(U',W') : U\not\subseteq W', U'\subseteq W \text{ or } U\subseteq W', U'\not\subseteq W\} \\
        \Delta_3 &= \{(U',W') : U\not\subseteq W', U'\not\subseteq W\}.
    \end{align*}
    We compute the intersection number $p_{22}^1$; the others are similar.
    Consider two pairs $(A,B)$ and $(A,C)$, both lying in $\Delta_1$. We have that $A\subseteq B$ and $A\subseteq C$, so that $A=B\cap C$. To compute $p_{22}^1$, we count the pairs $(X,Y)$ satisfying
\begin{itemize}
    \item $X\subseteq B$, $A\not\subseteq Y$ and $X\subseteq C$, $A\not\subseteq Y$; or 
    \item $X\not\subseteq B$, $A\subseteq Y$ and $X\not\subseteq C$, $A\subseteq Y$.
\end{itemize}
Consider the case where $A\subseteq Y$. We need $X\not\subseteq B$ and $X\not\subseteq C$. Since $Y$ is a two-space with $Y\supseteq X$, one sees $Y=\braket{X,A}$. The number of required $X$ is the number of one-spaces in $V$, which are not in $B$, or in $C$; the number of such choices for $X$ is $\qbinom{3}{1}-2\qbinom{2}{1}+1=q(q-1)$. In the second case $A\not\subseteq Y$, so we need $X\subseteq B$ and $X\subseteq C$, which implies $X=B\cap C=A$, so there are none in this case, and $p_{22}^1 = q(q-1)$.

For each $1\le i\le 3$, to check for a strongly regular orbital graph $\Gamma_{i}$, one needs only check for equality in the non-starred entries in column $i$ of the matrix $M$. Clearly, this only happens in the middle column when $q=4$.
\end{proof}

\subsubsection{Nonsingular subspace actions}

The intersection numbers of the $\mathrm{PGU}_n(q)$ nonsingular-point scheme for $n\ge3$ were determined by Wei \cite{Wei} and recorded by Bannai, Song, Shen and Wei \cite[Sections~1.1--1.3]{bannai-song-unitary-symplectic}. Their relations $R_1$ and $R_q$ correspond to our $\Delta_0$ and $\Delta_1$. We reproduce the required counts in this labelling for completeness, keeping the endpoint exclusions explicit; the passage to the socle in small dimension is justified below, and the rank four fusion for $q=4$ is performed below.

\begingroup
\setlength{\topsep}{\medskipamount}
\begin{prop}
    Let $G\in \mathcal{G}$ be the almost simple group in Table~\ref{tab:rank-four-families} with $S=\soc{G}=\U_n(q)$, $q\in \{3,4\}$, and point stabiliser $N_1$. Then the graph $NU_n(q)$ in Example~\ref{ex:graph-families} (i) is the only strongly regular orbital graph associated to $G$.
\end{prop}
\endgroup
\begin{proof}
Let $V$ be an $n$-dimensional vector space over $\F_{q^2}$ and $h : V\times V \to \F_{q^2}$ the associated non-degenerate hermitian form on $V$; define $f(v)=h(v,v)$. For $\alpha\in \F_{q^2}$, define $N: \F_{q^2}\to \F_q$ as $N(\alpha)=\alpha\bar{\alpha}=\alpha^{q+1}$, and let $K := \{\alpha\in \F_{q^2}^* : N(\alpha)=1\}$.
For $n\ge4$, the action of $S$ on nonsingular one-spaces has the same orbitals as that of $\mathrm{PGU}_n(q)$, by \cite[Lemma 1.30]{allen-thesis}. For $n=3$ and $q\in\{3,4\}$, we have $S=\mathrm{PGU}_3(q)$ since $\gcd(3,q+1)=1$. Thus the non-diagonal orbitals of $S$ are
\[
    \Delta_\lambda = \{(\braket{x},\braket{y}) : \braket{x}\ne\braket{y},\ h(x,x)=h(y,y)=1, N(h(x,y))=N(\lambda)\}
\]
where $\lambda$ ranges over a set $\Lambda\subseteq\F_{q^2}$ containing $0$ and $1$ such that $N:\Lambda\to\F_q$ is a bijection. Thus the relations are parametrised by the norm values $N(\lambda)\in\F_q$, with the diagonal kept separate.

We claim that the projective line joining $\braket{x}$ and $\braket{y}$ is tangent to the hermitian variety if and only if $(\braket{x},\braket{y})\in \Delta_1$. The projective line $L$ joining them is tangent precisely when it meets the variety $H=\{v\in V : f(v)=0\}$ in a single projective point. Choose the norm-one representatives so that $h(x,y)=\lambda$. For $v=x+\alpha y\in L$, with $\alpha\in \F_{q^2}$, we see that $v\in H$ if and only if $N(\alpha+\lambda)=-1+N(\lambda)$; the number of solutions $\alpha\in \F_{q^2}$ to this equation is $1$ if and only if $N(\lambda) = 1$, proving the claim. Hence the graph of $\Delta_1$ is precisely $NU_n(q)$. \vspace{1mm}

Next, we claim that none of the other orbital graphs is strongly regular. We first treat $q=3$. \vspace{1mm}

For $q=3$, choose $\Lambda=\{0,1,\omega\}$ with $N(\omega)=2$, so that the non-diagonal orbitals are $\Delta_0$, $\Delta_1$ and $\Delta_\omega$. Suppose that there is another strongly regular orbital graph $\Gamma_\rho$ corresponding to the $\Delta_\rho$ orbital with $N(\rho)\ne 1$, so $\rho \in \{0,\omega\}$. Let $\sigma\in \{0,\omega\}$ be such that $\rho \ne \sigma$, hence $p_{\rho\rho}^1 = p_{\rho\rho}^\sigma$; we show that this is impossible. We begin by computing the intersection number $p_{\omega\omega}^1$, which by Lemma~\ref{lem:example} is equal to $p_{1\omega}^\omega
\frac{k_\omega}{k_1}$, where $k_\lambda$ is the size of the suborbit indexed by $\lambda$; and so we compute $p_{1\omega}^\omega$, that is the number of $\braket{z}\subset V$ with $f(z)=1$ such that $N(h(x,z)) = 1$, and $N(h(z,y))=N(\omega)$. Let $\mathcal{B}=\{v_1,\ldots, v_n\}$ be an orthonormal basis for $V$, and choose a pair of vectors $(x,y)$ with $f(x)=1$, and $f(y)=1$, such that $N(h(x,y))=N(\omega)$, and $(\braket{x},\braket{y})\in \Delta_\omega$. To satisfy this, one may choose $x = v_1$, and $y = \omega v_1 + \alpha v_2$ to be such that $N(\omega)+N(\alpha)=1$.  Writing $z = \sum_i\, \alpha_iv_i$, gives $N(\alpha_1)=1$, for which there are $q+1$ choices of $\alpha_1$. It also implies that $N(\alpha_1\bar{\omega} + \alpha_2\bar{\alpha}) = N(h(z,y)) = N(\omega)$, for which there are $q+1$ choices of $\alpha_2$, one of which is $\alpha_2=0$. The fact that $f(z)=1$ implies that
\[
    \bsum_{i=3}^n \, N(\alpha_i) = -N(\alpha_2).
\]
For $r\ge1$, let $\mathcal{N}_r(c)$ with $c\in\{0,1\}$ denote the number of solutions to $\sum_{i=1}^r N(\alpha_i)=c$, including the zero vector when $c=0$. The standard norm counts, \cite[Theorem 26.9]{comb}, give
\[
    \mathcal{N}_n(0)=q^{2n-1}+(-1)^n(q^n-q^{n-1}),\qquad
    \mathcal{N}_n(1)=q^{2n-1}-(-1)^nq^{n-1}.
\]
Every prescribed nonzero norm has the same count, by scaling. Dividing the vector count by the $q+1$ norm-one representatives of each line and excluding $\braket{x}$ gives
\[
    p_{1\omega}^\omega=\mathcal{N}_{n-2}(0)+q\mathcal{N}_{n-2}(1)-1.
\]
We compute the suborbit lengths, $k_\lambda$, to be
\[
    k_0 = \frac{\mathcal{N}_{n-1}(1)}{q+1}, \quad
    k_1 = \mathcal{N}_{n-1}(0)-1, \quad
    k_\omega = \mathcal{N}_{n-1}(1),
\]
Hence
\[
    p_{\omega\omega}^1 = \frac{\mathcal{N}_{n-1}(1)}{\mathcal{N}_{n-1}(0)-1}
    \left(\mathcal{N}_{n-2}(0)+q\mathcal{N}_{n-2}(1)-1\right),
\]
and similarly, $p_{\omega\omega}^0 = (q+1)\mathcal{N}_{n-2}(0)$. Substituting $q=3$ gives
\[
    p_{\omega\omega}^1-p_{\omega\omega}^0
    =5(-3)^{n-3}\ne0,
\]
so the graph of $\Delta_\omega$ is not strongly regular. In a similar fashion we obtain
\[
    p_{00}^\omega = \frac{\mathcal{N}_{n-2}(1)}{q+1},\qquad
    p_{00}^1 = \frac{\mathcal{N}_{n-1}(1)(\mathcal{N}_{n-2}(0)-1)}
    {(q+1)(\mathcal{N}_{n-1}(0)-1)}.
\]
Here $p_{01}^0=\mathcal{N}_{n-2}(0)-1$ excludes the endpoint, and Lemma~\ref{lem:example} gives $p_{00}^1=p_{01}^0k_0/k_1$. Again substituting $q=3$ gives
\[
    p_{00}^1-p_{00}^\omega=-(-3)^{n-3}\ne0.
\]
Thus the graph of $\Delta_0$ is not strongly regular, and the tangent graph $NU_n(3)$ is the only strongly regular orbital graph.

For $q=4$, the socle action has rank five. Since $G$ has rank four, it contains an automorphism inducing $a\mapsto a^2$ on the norm values in $\F_4$, as does the Frobenius automorphism $x\mapsto x^2$ of $\F_{16}$. This interchanges the two values in $\F_4\setminus\{0,1\}$, fusing their relations into a single orbital relation $\mathcal O$. Proceeding similarly and summing the corresponding intersection numbers, we obtain
\[
    p_{00}^{1}-p_{00}^{\mathcal O}=-(-4)^{n-3}\ne0,\qquad
    p_{\mathcal O\mathcal O}^{0}-p_{\mathcal O\mathcal O}^{1}=3(-4)^{n-2}\ne0.
\]
Thus neither the orthogonality graph nor the graph on $\mathcal O$ is strongly regular, leaving the tangent graph as the unique strongly regular orbital graph.
\end{proof}

\begin{prop}
    Let $G$ be an almost simple group in Table~\ref{tab:rank-four-families} with socle $S=\Pommega_{2m+1}(5)$, where $m\ge2$, and socle stabiliser $N_1^\epsilon$, where $\epsilon\in\{\pm1\}$. Then its strongly regular orbital graphs are precisely the orthogonality graph $NO_{2m+1}^{\epsilon,\perp}(5)$ and the tangent graph $NO_{2m+1}^{\epsilon}(5)$, with the parameters given in Example~\ref{ex:graph-families} (ii).
\end{prop}
\begin{proof}
    Let $Q:V\to\F_5$ be the associated quadratic form, rescaled so that the one-spaces in $\Omega$ have norm-one representatives, and let $b(x,y)=Q(x+y)-Q(x)-Q(y)$ be its polar form. Recall from Section~\ref{subsec:orthogonal-intersection numbers} that, for $\lambda\in\{0,1,2\}$, the non-diagonal orbitals are given by
    \[
        \Delta_\lambda=\{(\braket{x},\braket{y}):\braket{x}\ne\braket{y},\
        Q(x)=Q(y)=1,\ b(x,y)^2=\lambda^2\}.
    \]
    Thus $\Delta_0$ is orthogonality. The projective line $\ell=\langle x,y\rangle$ joining $\left\langle x \right\rangle$ and $\left\langle y \right\rangle$ is tangent precisely when it meets the quadric in one point. An arbitrary point of $\ell$ other than $\langle y\rangle$ is represented by $x+\alpha y$. Thus tangency occurs when there is one value of $\alpha$ such that
    \[
        Q(x+\alpha y)=1+\alpha b(x,y)+\alpha^2 = 0.
    \]
    Thus tangency occurs precisely when the discriminant $b(x,y)^2-4$ vanishes, so $\Delta_2$ is tangency.

    Taking $\eta=\epsilon$ in Proposition~\ref{orth:prop:odd_dim_orthogonal_pijk} and letting $u=5^{m-1}$, we obtain
    \[
        p_{00}^{1}-p_{00}^{2}=0,\qquad
        p_{11}^{0}-p_{11}^{2}=-\epsilon u\ne0,\qquad
        p_{22}^{0}-p_{22}^{1}=0.
    \]
    Hence the orthogonality and tangent graphs are strongly regular, and no other orbital graphs are.
\end{proof}

\begin{prop}
    Suppose that $G$ is an almost simple group in Table~\ref{tab:rank-four-families} with socle $S=\Pommega_{2m}^\epsilon(q)$, where $m\ge3$, $\epsilon\in\{\pm1\}$ and $q\in\{4,5\}$, and socle stabiliser $N_1$. Then none of its orbital graphs are strongly regular.
\end{prop}
\begin{proof}
    Let $(V,Q)$ have dimension $2m$ and type $\epsilon$ over $\F_q$, with polar form $b$.

    First suppose that $q=5$. Recall that, for $\lambda\in\{0,1,2\}$, the three non-diagonal orbitals are
    \[
        \Delta_\lambda=\{(\braket{x},\braket{y}):\braket{x}\ne\braket{y},\
        Q(x)=Q(y)=1,\ b(x,y)^2=\lambda^2\}.
    \]
    Applying Proposition~\ref{orth:prop:even_dim_orthog_pijk} with $u=5^{m-1}$ gives
    \[
        p_{00}^{1}-p_{00}^{2}=-\frac{2\epsilon u}{5},\qquad
        p_{11}^{0}-p_{11}^{2}=\frac{3\epsilon u}{5},\qquad
        p_{22}^{0}-p_{22}^{1}=-\frac{4\epsilon u}{5}.
    \]
    None of these differences vanishes, so no orbital graph is strongly regular.

    Now suppose that $q=4$, and write $\F_4=\{0,1,\omega,\omega^2\}$, where $\omega^2+\omega+1=0$. For $\lambda\in\F_4$, the non-diagonal orbitals of $S=\Pommega_{2m}^\epsilon(4)$ are
    \[
        \Delta_\lambda=\{(\braket{x},\braket{y}):\braket{x}\ne\braket{y},\
        Q(x)=Q(y)=1,\ b(x,y)=\lambda\}.
    \]
    In the rank-four extension $G$, the orbital labels are sent to their squares, fusing $\Delta_\omega$ and $\Delta_{\omega^2}$. Thus the three non-diagonal orbitals of the rank four group $G$ are
    \[
        R_1=\Delta_0,\qquad
        R_2=\Delta_\omega\cup\Delta_{\omega^2},\qquad
        R_3=\Delta_1.
    \]
    Write $p'{}_{ij}^{k}$ for the intersection numbers of these fused relations. Summing the corresponding original intersection numbers from Proposition~\ref{orth:prop:even_char_general}, with $u=4^{m-1}$, gives
    \[
        p'{}_{11}^{2}-p'{}_{11}^{3}=\frac{\epsilon u}{2},\qquad
        p'{}_{22}^{1}-p'{}_{22}^{3}=-\epsilon u,\qquad
        p'{}_{33}^{1}-p'{}_{33}^{2}=\frac{3\epsilon u}{4}.
    \]
    Again all three differences are nonzero, so no single-orbital graph is strongly regular.
\end{proof}
This completes the proof of Theorem~\ref{thm:rank-four-srg}. \qed

\vspace{5cm}
\begin{center}
Department of Mathematics, Imperial College London, SW7 2AZ, UK 

will.allen21@imperial.ac.uk
\end{center}

\end{document}